\documentclass[journal]{assets/IEEEtran}

\usepackage[english]{babel}

\RequirePackage[T1]{fontenc}
\usepackage{graphicx}
\usepackage{amsfonts}
\usepackage{algorithm}
\usepackage{algorithmic}
\usepackage{booktabs}
\usepackage{multirow}
\usepackage{amsthm}
\newtheorem{definition}{Definition}
\newtheorem{theorem}{Theorem}
\newtheorem{proposition}{Proposition}
\usepackage{amsopn}

\usepackage{amsmath}
\usepackage{amssymb}
\usepackage{yhmath}
\usepackage{mathrsfs}
\usepackage{enumerate}
\usepackage{bm}
\usepackage{float}
\usepackage{placeins}

\usepackage{xcolor}
\usepackage{booktabs}
\usepackage{tabularx}
\usepackage{multirow}
\usepackage{amsmath,amssymb}
\usepackage{etoolbox}
\usepackage{threeparttable}

\usepackage{txfonts}

\usepackage[colorlinks=true,
linkcolor=blue,
citecolor=red,
urlcolor=magenta,
backref=page
]{hyperref}

\makeatletter
\def\@IEEEBIOskipN{0.25\baselineskip}
\patchcmd{\IEEEbiography}
  {\vskip \@IEEEBIOskipN plus 1fil minus 0\baselineskip}
  {\vskip \@IEEEBIOskipN plus 0.1\baselineskip minus 0.1\baselineskip}
  {}{}
\makeatother

\begin{document}
\title{LLCT: Learnable Linear Canonical Transform Family with Stability for Visual Feature Processing}

\author{Zunwei~Fu, \textit{Member, IEEE},
        Zhen~Wang, \textit{Member, IEEE},
        Dachun~Yang,
        and~Shuhui~Yang
\thanks{This work was supported in part by the National Natural Science Foundation of China (Grant Nos. 12431006, 12371093, and 12601171),  the Beijing Natural Science Foundation (Grant No. 1262011), the Fundamental Research Funds for the Central Universities (Grant No. 2253200028), the Shandong Natural Science Foundation (Grant No. 2025TSGCCZZB0890), and the National Research Foundation of Korea (Grant No. RS-2026-25474453).}
\thanks{Zunwei Fu is with the Research Center for Big Data and Artificial Intelligence, Linyi University, Linyi 276005, The People's Republic of China, and also with the Department of Computer Science, The University of Suwon, Hwaseong-si, Gyeonggi-do, Republic of Korea (e-mail: zwfu@suwon.ac.kr).}
\thanks{Zhen Wang is with the School of Mathematics and Statistics, Linyi University, Linyi 276005, The People's Republic of China (\mbox{e-mail:} zanewang.\allowbreak research@\allowbreak hotmail.com).}
\thanks{Dachun Yang is with the Laboratory of Mathematics and Complex Systems (Ministry of Education of China),
	School of Mathematical Sciences, Institute for Advanced Study,
	Beijing Normal University, Beijing 100875,
	The People's Republic of China (e-mail: dcyang@bnu.edu.cn).}
\thanks{Shuhui Yang is with the Department of Mathematics, Jinzhong University,  Jinzhong 030619,  The People's Republic of China (e-mail: yangshuhui1@jzxy.edu.cn).}}

\maketitle
\begin{abstract}
Visual structures are often characterized by directional geometry, local oscillation, and spatial--frequency coupling. To overcome the limitations of conventional spatial convolutions mainly operating on fixed Euclidean coordinates, this article develops the learnable linear canonical transform (LLCT) framework that learns parameters within the linear canonical transform (LCT) family. The framework treats the LCT family as parameterized analytical views spanning spatial, Fourier, and intermediate linear canonical representations. The framework has the following two complementary modes. The Linear Canonical Riesz Transform (LCRT) mode combines learned linear canonical coordinates with regularized Riesz-type directional responses at sites chosen for geometric and local-structural refinement. The learnable LCT mode provides the integrated responses for latent feature sites. We establish the parameter stability of the LCT representation under admissible perturbations of the canonical matrices. Both modes use chirp multiplication and the fast Fourier transform (FFT) to preserve feature dimensions and support automatic differentiation. Experiments on crack segmentation and single-image deraining evaluate the above two modes with multiple backbones. On four segmentation backbones, the full LCRT mode gives the highest reported point estimates of F1 and IoU for the crack class among the tested variants. The learnable LCT mode gives higher values in 30 of 32 paired PSNR and SSIM measurements across four deraining backbones and four datasets. In ablation studies, the full model has the highest reported PSNR and SSIM among five tested configurations.
\end{abstract}


\IEEEpeerreviewmaketitle

\section{Introduction}

\IEEEPARstart{D}{eep} visual features encode local structures such as oriented edges, contours, oscillations, and frequency shifts, and the expression of these structures varies across spatial locations and network stages. Under the manifold assumption, natural visual representations are commonly modeled as lying on a low-dimensional manifold embedded in a high-dimensional ambient space \cite{Li2025DenoisingGenerativeModels}. From this perspective, effective representations must not only capture the strength of local responses, but also organize their variation in coordinates, which makes the intrinsic structure explicit and compact \cite{Peng2025NatureMachineIntelligence}. Note that conventional convolutions describe local interactions on a fixed Euclidean grid \cite{LeCun2015DeepLearning}, whereas standard spectral operators express the same features in fixed frequency coordinates. Used either alone, these fixed coordinate systems do not provide an explicit mechanism for learning alternative analytical coordinate views at selected feature sites.

Analytical coordinates provide an explicit way to organize visual structure: a coordinate system determines how directional variation, local oscillation, and spatial--frequency relations are expressed, which in turn affects how readily downstream layers can distinguish and recombine task-relevant feature components.  For example, Fourier operators expose spectral composition, wavelets preserve multiscale locality, Riesz transforms encode phase and orientation, and learned spectral operators parameterize interactions in a selected transform domain \cite{Plonka2023NumericalFourierAnalysis,Unser2009Monogenic,Li2021FourierNeuralOperator}. These approaches demonstrate the value of transform coordinates, but prior work generally uses fixed or narrowly parameterized coordinates. At the same time, the feature hierarchy presents two distinct feature conditions. Some features, especially at shallow or transitional stages, retain explicit directional geometry that can be refined with a directional prior. Others contain mixed latent structure in which local patterns interact with factors associated with the task or degradation. These observations inspire a natural  representation question: whether visual features can be organized within a unified analytical coordinate family that spans spatial, Fourier, and intermediate spatial--frequency representations, while remaining compatible with structure-sensitive operators such as Riesz/directional wavelet analysis. Such a representation framework would move beyond a fixed transform domain by allowing the coordinate configuration itself to participate in feature representation. In this setting, the analytical coordinates are learned jointly with deep features, and their resulting responses are used for different feature-processing purposes at selected network sites. Learning the canonical parameters also motivates examining how the transform representation depends on parameter perturbations.
We develop LLCT, a visual feature-processing framework based on the linear canonical transform (LCT) family. The LCT defines a  family, parameterized via a canonical matrix $A$, of analytical coordinates, which includes the spatial and the Fourier representations as endpoint configurations, together with intermediate configurations that introduce chirp modulation, scaling, or coupling between spatial and frequency coordinates \cite{Zayed2018LCTDiscSampling,Chen2024GraphLCT,He2025SymplecticWignerLCT}. Each canonical matrix $A$ specifies the geometry of a coordinate view. At each selected feature site \(s\), a learned matrix pair \(\boldsymbol A^{(s)}\) parameterizes the corresponding linear canonical view. Fig.~\ref{fig:newfig1} illustrates how different canonical matrices reexpress the same feature field.

\begin{figure}[t]
  \centering
  \includegraphics[width=\linewidth]{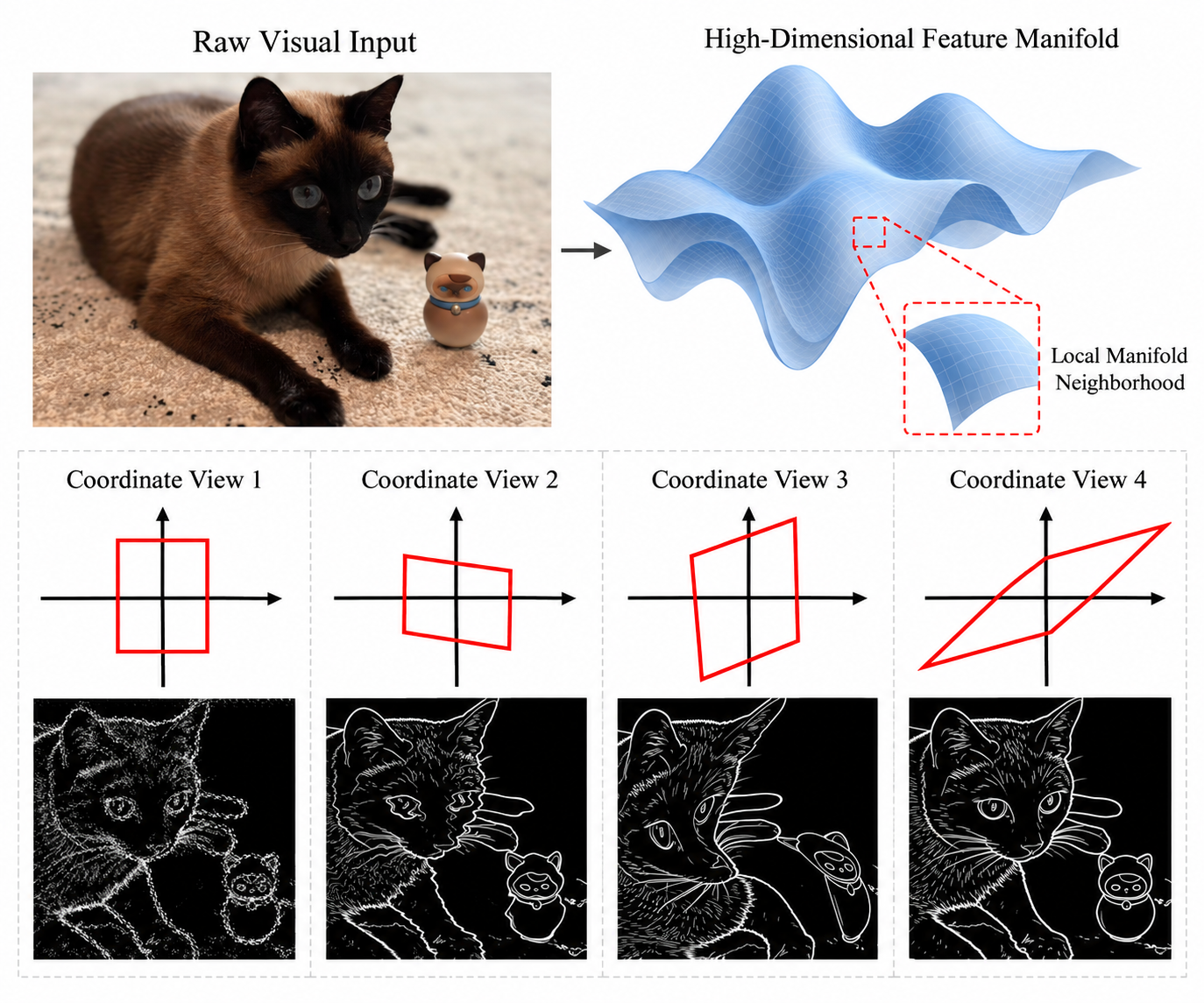}
  \caption{Different coordinate views in the LCT family reexpress the same local feature field while emphasizing different structural relations.}
  \label{fig:newfig1}
\end{figure}

 LLCT instantiates this transform family through two complementary modes: the linear canonical Riesz transform (LCRT) mode and the learnable LCT mode. In the LCRT mode, a Riesz-type directional multiplier operates within learned linear canonical coordinates, and produces responses sensitive to orientation for features in which edges, contours, and thin structures remain spatially explicit. The learnable LCT mode combines responses from spatial, Fourier, and learned linear canonical coordinates. Real and imaginary coefficients are projected jointly, allowing the learned response to depend on transform phase. The balance between the two endpoint responses regulates the contribution of the learned canonical branches. This mode targets latent representations whose structures depend jointly on space and frequency. Both modes convert their analytical responses into real-valued tensors with the same dimensions as the host features, and introduce the resulting correction through a common gated residual interface. This interface retains the original feature pathway, backbone blocks, prediction head, task objective, and input and output formats. To sum up, the main contributions of this article are as follows.

\begin{enumerate}[(i)]
  \item We formulate visual feature reparameterization through learnable LCTs. This formulation not only interprets the LCT as a parameterized family of connected analytical views rather than merely as an auxiliary transform-domain operator, but also provides the coordinate basis for the two modes developed in this article. We further establish the parameter stability of the continuous LCT, expressed as $L^2$-continuity with respect to the canonical matrices for each fixed Schwartz-class input on the nondegenerate parameter domain, as stated in Theorem~\ref{thm:lct-parameter-stability}.
  
  \item We implement LLCT using a canonical parameterization and piecewise differentiable discrete operations, supporting end-to-end, gradient-based optimization of the canonical parameters jointly with the host-network weights. The LCRT mode is deployed at selected shallow sites for implementing Riesz-type directional analysis, while the learnable LCT mode is deployed at selected latent or restoration sites to combine responses from spatial, Fourier, and learned linear canonical coordinates. Parameters are shared according to the selected site and the structure of the host block, and both modes remain compatible with existing deep architectures.
  
  \item Across multiple backbones, we respectively evaluate the LCRT mode on crack segmentation and the learnable LCT mode on image deraining. The matched comparisons examine feature reparameterization with the LCT family in the segmentation of thin structures and in image restoration. The observed improvements across these settings support the effectiveness of this framework across different visual tasks and backbone architectures.
\end{enumerate}

\section{Related Work}

The work most relevant to this study can be summarized into three closely related aspects, whose details are presented, respectively, in Subsections \ref{RW-A}, \ref{RW-B}, and \ref{RW-C}: analytical coordinate representations, Riesz-type structural analysis, and the use of structured transforms in learned visual models. Analytical transforms determine the coordinate systems in which feature variation is represented, Riesz-type operators make directional and phase-related structure explicit within those coordinates, and transform-based neural models examine how such analytical representations can be incorporated into feature learning. These three lines of work provide the background for learning linear-canonical coordinates and using their responses for visual feature reparameterization.

\subsection{Analytical Coordinate Families for Visual Representation}\label{RW-A}

Analytical transforms provide various coordinate families for signal representation. The Fourier transform gives a fixed spectral view, whereas the fractional Fourier transform (FrFT) provides a family, controlled by transform order, of intermediate representations between the spatial and the frequency domains \cite{Plonka2023NumericalFourierAnalysis,McBride1987NamiasFrFT}. The LCT broadens this family by parameterizing the transform through canonical matrices, whose different configurations yield different analytical coordinate views. Particularly, Spatial, Fourier,  fractional Fourier, and Fresnel representations correspond to particular configurations, while other configurations describe scaling, chirp modulation, and more general coupling between spatial and frequency coordinates \cite{Zayed2018LCTDiscSampling,Fu2027ACHA, Bhandari2012FrFTSampling}. This flexibility is particularly effective when nonstationary patterns are not well represented in a single fixed coordinate system. 

Sampling and reconstruction theories have been developed for signals bandlimited in the linear canonical transform domain \cite{Tao2008LCTSampling}. Beyond these theoretical developments, recent work has investigated parameter selection, feature construction, and filtering in generalized transform domains. Zhang \cite{Zhang2019LCWDLFMSNR} selected the free parameters of a linear canonical Wigner distribution through an output-SNR analysis for noisy LFM signal detection. More recently, Zhang et al. \cite{Zhang2025BroadLearningFractionalFeature} integrated features from multiple fractional Fourier orders, together with their interactions, into a broad learning system. For graph signals, Cui et al. \cite{Cui2025GraphChirpGFED} introduced graph chirp signals and a graph fractional vertex--frequency energy distribution for detection and denoising, and Chen et al. \cite{Chen2024GraphLCT} defined the graph linear canonical transform and considered graph-domain filters designed analytically or learned by stochastic gradient descent. These studies demonstrate the value of adapting generalized transform representations to particular signal models and objectives. Our formulation differs in both scope and mechanism, which learns site-specific canonical coordinate configurations end to end at selected locations within a visual feature hierarchy.

\subsection{Riesz-Type Transforms in Generalized Coordinates}\label{RW-B}

Riesz analysis mainly concerns the problem: making directional relations explicit after a coordinate system has been chosen. The classical Riesz transform extends analytic signal analysis to multiple dimensions. Its combination with an image forms the monogenic signal, which represents local amplitude, phase, and orientation \cite{Felsberg2001MonogenicSignal}. Riesz--Laplace wavelets place the same principle in a steerable multiresolution construction, and recover local structure across scale and orientation \cite{Unser2009Monogenic}. Riesz-type transforms thus characterize the directional geometric features of images through explicit analytical responses. 

Generalized Riesz operators alter the coordinates in which this directional analysis is performed. Recently, the fractional Riesz transform is defined through a chirp multiplier in terms of the FrFT, and its order controls the amplitude, phase, and directional information recovered in edge analysis \cite{Fu2023RieszFrFT}. Fractional Riesz potentials extend related operators associated with chirp modulation to other image processing settings, although they constitute a distinct operator family \cite{Fu2024FrFTRieszPotentials}. More recently, Yang et al. \cite{Yang2026LCRTImageEdgeDetection} extended Riesz transforms to linear canonical coordinates through the LCRT, and established their characterizations in terms of  multipliers. This work provides the mathematical foundation of the LCRT used in this article, but the article \cite{Yang2026LCRTImageEdgeDetection} only studies the analytical edge detector rather than learning canonical coordinates inside a neural feature hierarchy. Learning such coordinates jointly with deep features is a central methodological objective of the present article.

In addition, several learned models use representations derived from Riesz-type transforms. The Riesz-Quincunx-UNet integrates the Riesz-wavelet construction into a variational autoencoder for the denoising of satellite images \cite{Thai2023RieszQuincunxUNet}. A scale-equivariant Riesz representation has likewise been used in the hybrid deep models for classification \cite{Barisin2024RieszFeatureRepresentation}. In both cases, Riesz analysis contributes directly to the learned representation. In these aforementioned models, Riesz-type transforms are tied to a specific multiscale or architectural design, whereas our work learns linear canonical coordinates at selected network sites for different feature processing roles. 

\subsection{Analytical Transforms in Learned Visual Representations}\label{RW-C}

Structured transforms enter neural models as prescribed operators or parameterized components. Fourier neural operators provide another form of parameterization by learning spectral multipliers for mappings between function spaces \cite{Li2021FourierNeuralOperator}. Here the analytical structure is either fixed or partly learned, but the objective is invariance, the parameterization of wavelet filters, or operator learning rather than learning linear canonical coordinates for feature reparameterization at selected sites in a visual backbone. Trainable fractional transforms provide the closest context for the present work. Koç et al. \cite{Koc2024LSP} showed that the fractional order of the discrete FrFT can be optimized jointly with network weights through backpropagation. Yu et al. \cite{Yu2023NeurIPS} introduced multi-order fractional Fourier convolution, combining spatial, Fourier, and intermediate fractional-domain processing for visual tasks. These establish transform-parameter learning and multidomain feature processing. Different from them, the distinction in our formulation lies in the transform family being
learned. An FrFT order controls rotation in the spatial--frequency plane, whereas each canonical matrix in our model is parameterized by rotation, scaling, and shearing variables. Consequently, the learned coordinates can vary in both orientation and shape, with FrFT rotations retained as a special case.

Image restoration offers some  closer examples of feature processing in transform domains. FADformer combines the spatial and the Fourier domain operations for image deraining \cite{Gao2024FADformer}. Such results support processing in the frequency domain for image restoration, but the transform coordinates remain based on the Fourier transform, and the modules are designed for individual image restoration tasks. Their focus remains on task-specific processing in fixed transform coordinates, rather than on learning a parameterized canonical coordinate system for different representation roles at selected feature sites, different from the concerns of the present article.

\bigskip

To sum up, these studies show the value of analytical transforms in feature learning, while the use of learnable analytical coordinates within a visual feature hierarchy remains less explored. Some methods parameterize a coordinate or a filter, others apply Riesz-type analysis, and models designed for individual image tasks build transform blocks around a particular architecture. Building on these studies, we design LLCT, which combines the parameter continuity analysis of the continuous LCT with discrete modules that learn canonical parameters for visual feature processing. The LCRT mode uses its learned coordinates for Riesz-type directional refinement, while the learnable LCT mode combines responses from its learned coordinates with spatial and Fourier responses to explore latent features.

\section{Mathematical Foundations of LCT and LCRT}

This section fixes the transform conventions and notation used in the proposed framework. It defines the LCT coordinate family, establishes parameter stability on the stated domain in Subsection \ref{MF-A}, gives the LCRT multiplier formulation in Subsection \ref{MF-B}, and interprets the geometry of the canonical matrices needed to distinguish their mathematical roles in Subsection \ref{MF-C}. The LCRT results used below are taken from~\cite{Yang2026LCRTImageEdgeDetection}, and Proposition~\ref{prop:finite-grid-realization} of Appendix~\ref{bb} gives the finite-grid realization with its sensitivity bound.

\subsection{Linear Canonical Transforms and Parameter Stability}\label{MF-A}

We use the following Fourier transform convention. Let $\mathscr{S}(\mathbb{R}^n)$ denote the Schwartz space. For any $f \in \mathscr{S}(\mathbb{R}^n)$, its \emph{Fourier transform} $\mathscr{F}(f)$ or $\widehat{f}$ is defined for any $\boldsymbol{x} \in \mathbb{R}^n$ by
\begin{equation*}
  \widehat{f}(\boldsymbol{x})
  := \mathscr{F}(f)(\boldsymbol{x})
  := \frac{1}{(\sqrt{2\pi})^n}
  \int_{\mathbb{R}^n} f(\boldsymbol{y})
  e^{-i\boldsymbol{x}\cdot\boldsymbol{y}}
  d\boldsymbol{y},
\end{equation*}
and, following the notation for inverse transforms used in operator learning in the Fourier domain \cite{Li2021FourierNeuralOperator}, we denote
$
  f^\vee(\boldsymbol{x})
  :=
  \mathscr{F}^{-1}(f)(\boldsymbol{x})
  :=
  \widehat{f}(-\boldsymbol{x}).
$
With this notation, a Fourier neural operator (FNO) layer parameterizes its nonlocal kernel term through a learned Fourier multiplier,
\begin{equation}
  \label{eq:fno_fourier_multiplier}
  \mathcal{K}_{\phi}v
  =
  \mathscr{F}^{-1}
  \left(
  R_{\phi}\cdot\mathscr{F}v
  \right),
\end{equation}
where $R_{\phi}$ acts on selected spectral modes \cite{Li2021FourierNeuralOperator}. Equation~\eqref{eq:fno_fourier_multiplier} illustrates the parameterization of a multiplier in the transform domain followed by an inverse transform. The LCRT response in \eqref{eq:LCRT_multiplier} follows this operator pattern in LCT coordinates. The learnable LCT branch instead retains forward LCT coefficients and converts them into real response tensors for feature fusion, as specified in Section~\ref{siv}.

The LCT generalizes the fixed Fourier coordinate system to a family parameterized by canonical matrices, including spatial, Fourier, and intermediate linear canonical configurations \cite{Li2006HilbertLCT}.

Let $\boldsymbol{A}:=(A_1,A_2)$ with
$A_k:=\begin{bmatrix} a_k & b_k\\ c_k & d_k \end{bmatrix}\in{M_{2 \times 2}}(\mathbb{R}
  )$ ($M_{2\times 2}$ denotes the set of all $2\times 2$
real matrices with determinant equal to 1) for any $k\in\{1,2\}$, and let
$\boldsymbol{a}:=(a_1, a_2)$,
$\boldsymbol{b}:=(b_1, b_2)$,
and $\boldsymbol{d}:=(d_1, d_2)$.
We use this notation throughout.
The two-dimensional LCT on $\mathbb{R}^2$ is defined as follows.

\begin{definition}\label{2DLCT}
  For any $f \in \mathscr{S}({\mathbb{R}^2})$,
  the \emph{linear canonical transform}
  ${\mathscr{L}_{\boldsymbol{A}}}$
  is defined by setting, for any $\boldsymbol{u}\in\mathbb{R}^2$,
  \begin{equation}\label{eq:2d_lct}
    \mathscr{L}_{\boldsymbol{A}}(f)(\boldsymbol{u})
    :=
    \int_{\mathbb{R}^2}
    f(\boldsymbol{x})K_{\boldsymbol{A}}(\boldsymbol{x},\boldsymbol{u})
    \,d\boldsymbol{x},
  \end{equation}
  where, for any $\boldsymbol{x}:=(x_1,x_2),
    \boldsymbol{u}:=(u_1,u_2)\in \mathbb{R}^2$,
  \begin{equation*}
    K_{\boldsymbol{A}}(\boldsymbol{x},\boldsymbol{u})
    :=
    \prod_{k=1}^{2} K_{A_k}(x_k,u_k).
  \end{equation*}
  For each $k=1,2$, the kernel $K_{A_k}(x_k,u_k)$ is defined as
  \begin{equation*}
    K_{A_k}(x_k,u_k)=
    \begin{cases}
      \begin{aligned}
         & C_{A_k}\,e_{b_k,a_k}(x_k)                        \\
         & \quad\times e_{-b_k}(x_k,u_k)\,e_{b_k,d_k}(u_k),
      \end{aligned}
       & b_k\neq 0, \\[1.2ex]
      \sqrt{d_k}e^{\frac{i}{2}\left( {c_k}{d_k}u_k^2\right)}\delta(x_k-d_ku_k),
       & b_k=0,
    \end{cases}
  \end{equation*}
  with $C_{A_k}:=(i2\pi b_k)^{-1/2}$, $e_{b_k,a_k}(x_k):=e^{i a_kx_k^2/(2b_k)}$,
  $e_{-b_k}(x_k,u_k):=e^{-iu_kx_k/b_k}$, and
  $e_{b_k,d_k}(u_k):=e^{i d_ku_k^2/(2b_k)}$,
  where $\delta$ denotes the Dirac measure at $0$. All complex square roots are taken on the principal branch.

  Let $\boldsymbol{A}^{-1}:=(A_1^{-1},A_2^{-1})$,
  its \emph{inverse LCT}
  $({\mathscr{L}_{\boldsymbol{A}}})^{-1}(f)$
  in ${\mathbb{R}^2}$ is defined by setting, for any $\boldsymbol{x}\in \mathbb{R}^2$,
  $ \left({\mathscr{L}_{\boldsymbol{A}}}\right)^{-1}(f)
    (\boldsymbol x):={\mathscr{L}_{\boldsymbol{A}^{-1}}}(f)
    (\boldsymbol x).$
\end{definition}

Definition~\ref{2DLCT}, specifically \eqref{eq:2d_lct}, yields normalized forms of the Fourier transform, fractional Fourier transform, scaling operator, Fresnel transform, and chirp multiplication for particular choices of $\boldsymbol{A}:=(A_1,A_2)$. Transforms parameterized by analytical coordinates have been used in image and signal processing and in related problems in electromagnetics and wave propagation (see, for example, \cite{Abbas2017DFTPolarSphericalGrid,Bi2007WatermarkingEMD,Shen2004BiorthogonalWaveletReconstruction}). The LCT therefore provides a unified family of analytical coordinates that contains several classical representation forms as special cases. Here, the parameter stability refers to the \(L^2\)-continuity of the continuous LCT representation with respect to the canonical matrices for each fixed input in \(\mathscr S(\mathbb R^2)\). The following theorem establishes this property on the nondegenerate parameter domain \(b_k\ne0\), \(k\in\{1,2\}\).

\begin{theorem}\label{thm:lct-parameter-stability}
  Let
  $\boldsymbol A=(A_1,A_2)$ and
  $\boldsymbol A^{(m)}=(A_1^{(m)},A_2^{(m)})$, $m\in\mathbb N$,
  be canonical-matrix pairs, with
  $A_k=(\begin{smallmatrix}a_k&b_k\\c_k&d_k\end{smallmatrix})$ and
  $A_k^{(m)}
  =(\begin{smallmatrix}
    a_k^{(m)}&b_k^{(m)}\\
    c_k^{(m)}&d_k^{(m)}
  \end{smallmatrix})$
  for $k\in\{1,2\}$.
  Assume that
  $\det A_k=\det A_k^{(m)}=1$,
  $b_kb_k^{(m)}\neq0$, and
  $\boldsymbol A^{(m)}\to\boldsymbol A$ entrywise.
  Then, for every $f\in\mathscr{S}(\mathbb R^2)$,
  \begin{equation*}
    \lim_{m\to\infty}
    \left\|
    \mathscr{L}_{\boldsymbol A^{(m)}}(f)
    -
    \mathscr{L}_{\boldsymbol A}(f)
    \right\|_{L^2(\mathbb R^2)}
    =0.
  \end{equation*}
  Consequently, for each fixed
  $f\in\mathscr{S}(\mathbb R^2)$, the map
  $\boldsymbol A\mapsto\mathscr{L}_{\boldsymbol A}(f)$
  is continuous from
  $
    \Omega=
    \{
    (A_1,A_2):
    \det A_k=1,\ b_k\neq0,\ k=1,2
    \}
  $
  with the relative Euclidean topology into
  $L^2(\mathbb R^2)$.
\end{theorem}

The proof of Theorem~\ref{thm:lct-parameter-stability} is presented in Appendix \ref{aa}.

\begin{figure*}
  \centering
  \includegraphics[width=0.9\linewidth]{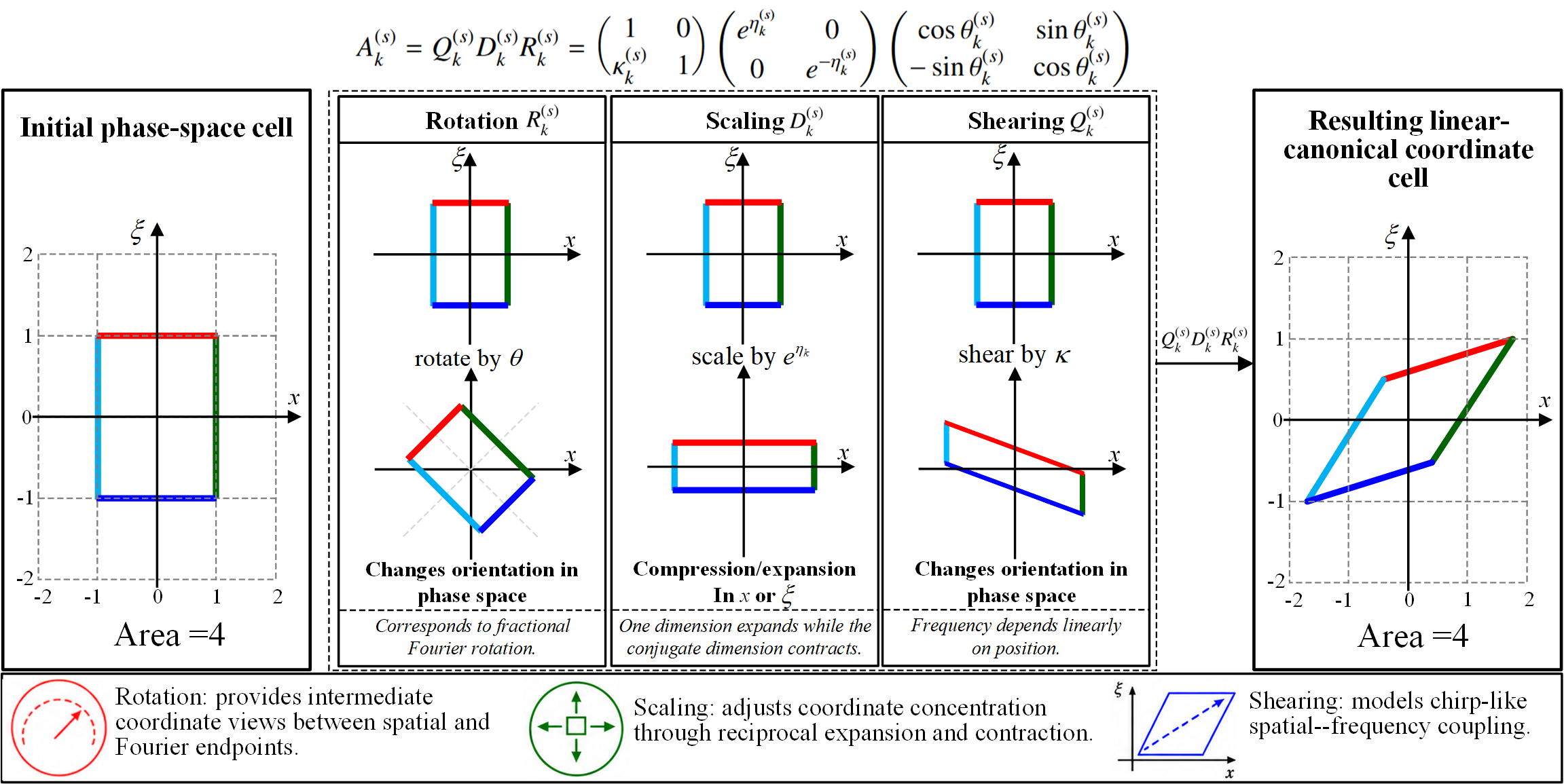}
  \caption{Geometric interpretation of the rotation--scaling--shearing factorization of an LCT canonical matrix.}
  \label{fig:lctvis}
\end{figure*}

\subsection{Multiplier Representation of the LCRT}\label{MF-B}

While the LCT provides a parameterized family of analytical coordinates for representation learning, it does not by itself isolate direction. We therefore use the LCRT, a Riesz-type directional operator in LCT coordinates (see, for example, \cite{Unser2009Monogenic,Felsberg2001MonogenicSignal,Fu2023RieszFrFT}).
For the subfamily satisfying $a_k=d_k$ and $b_k\neq0$, $k\in\{1,2\}$, Definition~\ref{def:LCRT} records the equivalent multiplier form in the LCT domain. Section~\ref{FR-A} uses a regularized version of this multiplier as the generalized discrete response for unrestricted learned canonical matrices.
\begin{definition}\label{def:LCRT}
  Let $j\in\{1,2\}$, and assume $a_k=d_k$ and $b_k\neq0$ for
  $k\in\{1,2\}$. The $j$-th {\rm LCRT}
  $R_{j}^{\boldsymbol A}$ has the following multiplier representation
    $
    -i\frac{\omega_j/b_j}
    {\left|\boldsymbol{\omega}/\boldsymbol b\right|}
  $ in the LCT domain.

  For any $f\in\mathscr{S}(\mathbb{R}^2)$ and
  $\boldsymbol{\omega}\in\mathbb{R}^2$, one has
  \begin{equation}\label{eq:LCRT_multiplier}
    \mathscr{L}_{\boldsymbol A}
    \!\left[R_{j}^{\boldsymbol A}(f)\right](\boldsymbol{\omega})
    =
    -i\frac{\omega_j/b_j}
    {\left|\boldsymbol{\omega}/\boldsymbol b\right|}
    \,\mathscr{L}_{\boldsymbol A}(f)(\boldsymbol{\omega}).
  \end{equation}
  Equivalently,
  \begin{equation}\label{eq:LCRT_impl}
    R_j^{\boldsymbol A}f(\boldsymbol x)
    =
    \mathscr{L}_{\boldsymbol A^{-1}}
    \!\left[
      -i\frac{\omega_j/b_j}
      {\left|\boldsymbol{\omega}/\boldsymbol b\right|}
      \,\mathscr{L}_{\boldsymbol A}(f)(\boldsymbol{\omega})
    \right](\boldsymbol x).
  \end{equation}
  where $\frac{\boldsymbol{\omega}}{\boldsymbol b}:=(\frac{\omega_1}{b_1},\frac{\omega_2}{b_2})$.
  The equality is understood in the distributional sense.
\end{definition}
Under the stated subfamily condition, the multiplier form in \eqref{eq:LCRT_multiplier}, followed by the inverse LCT in \eqref{eq:LCRT_impl}, implements the operator without numerical evaluation of the singular integral. In the discrete implementation used here, the zero-frequency component is set to zero and the denominator is regularized by a small positive constant.

For edge detection, \cite{Yang2026LCRTImageEdgeDetection} defines a sharpness measure associated with LCRT. For $j=1,2$, when $b_jc_j\neq0$, this quantity is defined as $R^{\rm E}_{\rm sc}(A_j)={b_j}/{c_j}.$

\subsection{Geometry of Canonical Matrices and Roles of LCT and LCRT}\label{MF-C}

Definition~\ref{2DLCT} makes the coordinate role explicit: the canonical matrix $A_k$ parameterizes the LCT kernel and thereby its coordinate system.

Fig.~\ref{fig:lctvis} shows how each canonical matrix
$A_k=
  \left(\begin{smallmatrix}
    a_k & b_k \\
    c_k & d_k
  \end{smallmatrix}\right)$
acts through a composition of elementary operations in phase space. In particular, $A_k$ can be reparameterized as
\begin{equation}
  \label{eq:lct_factorization}
  A_k
  =
  \begin{pmatrix}
    1        & 0 \\
    \kappa_k & 1
  \end{pmatrix}
  \begin{pmatrix}
    s_k & 0     \\
    0   & 1/s_k
  \end{pmatrix}
  \begin{pmatrix}
    \cos\theta_k  & \sin\theta_k \\
    -\sin\theta_k & \cos\theta_k
  \end{pmatrix},
\end{equation}
where
$s_k:=(a_k^2+b_k^2)^{1/2}$,
$\cos\theta_k:=a_k/s_k$,
$\sin\theta_k:=b_k/s_k$, and
$\kappa_k:=(a_kc_k+b_kd_k)/(a_k^2+b_k^2)$.
Equivalently, $\theta_k=\operatorname{atan2}(b_k,a_k)$, where $\operatorname{atan2}$ determines the quadrant of $\theta_k$.

The decomposition  \eqref{eq:lct_factorization} reparameterizes the existing LCT matrix without introducing independent transform parameters. Fig.~\ref{fig:lctvis} depicts a generic canonical block, in which we demonstrate the role of each matrix appearing in the decomposition  \eqref{eq:lct_factorization}. Read from right to left as a sequence of actions, the three matrices  correspond, respectively, to rotation, scaling, and shearing. The rotation matrix gives a reorientation analogous to the fractional Fourier transform, the scaling matrix adjusts the concentration of responses along conjugate coordinates, and the shearing matrix couples position and frequency through a chirp function. Since each matrix has unit determinant, the composition preserves area in phase space while changing the orientation and shape of the local cell.

\bigskip

The factorization characterizes the coordinate geometry, and directional selection is supplied separately by the LCRT multiplier operator. The LCT specifies the canonical coordinate system, whereas LCRT determines directional features with respect to those coordinates. The next section instantiates these complementary roles in the proposed framework for feature reparameterization.

\section{LLCT for Visual Feature Processing}\label{siv}

LLCT processes visual features through gated residual corrections computed from LCT-family responses. We refer to this operation as feature reparameterization. Specifically, in Subsection \ref{FR-A} we first introduce a common canonical parameterization and a
shape-preserving gated residual interface that incorporates
LCT-family responses while retaining the original feature pathway. In Subsection \ref{FR-B}, within this interface, the LCRT mode constructs directional refinements in learned linear-canonical coordinates, whereas the learnable LCT mode combines spatial, Fourier, and adaptive canonical responses to reorganize latent feature structure. We then specify the implementation, the parameter-sharing, and the initialization rules in Subsection \ref{FR-C}, and computational cost of the two modes in Subsection \ref{FR-D}.

\subsection{Canonical Parameterization and Residual Interface}\label{FR-A}

Let $s$ denote a selected insertion in a host backbone, and let 
$\mathbf{X}^{(s)}  \in  \mathbb{R}^{B\times C_s\times H_s\times W_s}$
denote the feature tensor at that site, where $B$ is the batch size, and $C_s$, $H_s$, and $W_s$
denote, respectively, the number of channels, height, and width at the site $s$. The mode and the insertion site are fixed by the experimental configuration rather than selected dynamically for each input. The LCT operation is specified for each selected site, and its parameters are learned jointly with the host network. When the host architecture reuses a block across recurrent steps, the inserted module follows the same rule for sharing parameters, and, otherwise, different insertion sites use separate parameters.

At each selected site $s$, either of LCT and LCRT modes learns
two $2\times2$ canonical matrices, one for each spatial axis, denoted by $\boldsymbol A^{(s)}=(A_1^{(s)},A_2^{(s)})$. 
Following \eqref{eq:lct_factorization}, for $k\in\{1,2\}$ each matrix $A_k^{(s)}$ is parameterized by 
\begin{equation}
	\label{eq:stage_canonical_parameterization}
	A_k^{(s)}  =Q_k^{(s)}D_k^{(s)}R_k^{(s)},
\end{equation}
where $Q_k^{(s)}=\left(\begin{smallmatrix}1&0\\ \kappa_k^{(s)}&1\end{smallmatrix}\right),$
    $D_k^{(s)}=\left(\begin{smallmatrix}e^{\eta_k^{(s)}}&0\\0&e^{-\eta_k^{(s)}}\end{smallmatrix}\right), $ and
    $R_k^{(s)}=\left(\begin{smallmatrix}
      \cos\theta_k^{(s)}  & \sin\theta_k^{(s)} \\
      -\sin\theta_k^{(s)} & \cos\theta_k^{(s)}
    \end{smallmatrix}\right).$
This parameterization enforces $\det A_k^{(s)}=1$. These matrices are parameterized by trainable \(\kappa_k^{(s)}\), \(\eta_k^{(s)}\), and \(\theta_k^{(s)}\), and the effect of each parameter on the module output depends on
the response obtained by the selected mode. At a given insertion site $s$, the same matrix pair  $\boldsymbol A^{(s)}$ is applied to every sample and channel, while $A_1^{(s)}$ and $A_2^{(s)}$ remain specific to the two tensor axes. For numerical evaluation of the discrete kernel and multiplier, in each occurrence of $b_k^{(s)}$ in a denominator
we replace it by $  \widetilde b_k^{(s)}:
  =
  \operatorname{sgn}(b_k^{(s)})
  \max(|b_k^{(s)}|,b_{\min}),$
where $b_{\min}=10^{-3}$ and
$\operatorname{sgn}(0)=1$, to avoid the denominator from becoming zero.

We use the notation $\mathscr{L}_{\boldsymbol A}$ from Section~\ref{MF-A} for a separable finite-grid realization that supports automatic differentiation. Its discrete operator performs chirp multiplication and orthonormal FFTs along the width and height axes. It operates independently on each channel of each sample, and returns a complex tensor with the same dimensions as its input. The hard floor makes the realization piecewise differentiable with respect to the canonical parameters. The learnable LCT mode retains only forward transform responses, whereas the LCRT mode can  also evaluate the inverse
LCTs to construct its directional response. A compatible finite-grid realization and sensitivity bound are given in Proposition~\ref{prop:finite-grid-realization} of Appendix~\ref{bb}. The stability analysis concerns the continuous LCT, while the discrete realization is piecewise differentiable with respect to its canonical parameters.

At an insertion site, the selected module mode
$\mathcal{M}_{\Theta^{(s)}}^{(s)}$
produces a residual correction
\begin{equation}
  \Delta\mathbf{X}^{(s)}
  =
  \mathcal{M}_{\Theta^{(s)}}^{(s)}
  \left(\mathbf{X}^{(s)}\right),
  \label{eq:delta_X_s}
\end{equation}
The correction in \eqref{eq:delta_X_s} gives the output
\begin{equation}
  \mathbf{Y}^{(s)}
  =
  \mathbf{X}^{(s)}
  +
  \gamma^{(s)}\odot\Delta\mathbf{X}^{(s)}.
  \label{eq:outputY_s}
\end{equation}
Here,
$\gamma^{(s)}\in\mathbb{R}^{C_s}$
is a learnable gate defined for each channel, initialized to zero, and broadcast over the batch and spatial dimensions. The zero initialization preserves the host feature path at the start of training. The parameter set $\Theta^{(s)}$ contains the canonical parameters, response projections, fusion operators, residual gates, and, for the learnable LCT mode, the parameters used for endpoint conditioning.

\subsection{Directional Reparameterization via LCRT}\label{FR-B}

The LCRT mode is deployed at insertion sites selected for directional refinement. Given $\mathbf X^{(s)}$, the learned response in linear canonical coordinates is $  \mathbf Z^{(s)}
  =
  \mathscr{L}_{\boldsymbol A^{(s)}}
  \left(\mathbf X^{(s)}\right)
  \in
  \mathbb C^{B\times C_s\times H_s\times W_s}.$

For general learned canonical matrices, we use the following multiplier as an operational definition of the generalized discrete Riesz response in the LCT domain:
\begin{align}\label{d1}
  \bm{\mathcal R}^{(s)}_j
  =
  \mathscr{L}_{(\boldsymbol A^{(s)})^{-1}}
  \!\left[
    -i\,
    \frac{\omega_j/\widetilde b_j^{(s)}}
    {\left|
      \boldsymbol\omega/\widetilde{\boldsymbol b}^{(s)}
      \right|+\varepsilon_{\mathrm R}}
    \mathbf Z^{(s)}
    \right],
  \qquad j\in\{1,2\}.
\end{align}
The zero-frequency multiplier is defined to be zero. For the subfamily satisfying $(a_k=d_k)$, \eqref{d1} is a regularized discrete realization of the established LCRT multiplier. When the numerical floor is inactive and $(\varepsilon_{\mathrm R}\to0^+)$, \eqref{d1} reduces to \eqref{eq:LCRT_multiplier} away from the origin. Moreover, the two components  $\{\bm{\mathcal R}^{(s)}_1, \bm{\mathcal R}^{(s)}_2\}$ in \eqref{d1} describe variation, respectively, along the two separable tensor directions under the learned canonical parameterization. 

We measure the directional magnitudes relative to the input feature using the \textit{phase-like descriptor}
\begin{equation}
  \mathbf{\Phi}^{(s)}:
  =
  \tan^{-1}\!\left[
    \frac{
      \sqrt{
        \left|\bm{\mathcal{R}}^{(s)}_1\right|^2
        +
        \left|\bm{\mathcal{R}}^{(s)}_2\right|^2
        +\varepsilon_{\Phi}
      }}
    {\mathbf{X}^{(s)}+\varepsilon_{\Phi}}
    \right],
  \label{eq:phi_s}
\end{equation}
where all operations in \eqref{eq:phi_s} are elementwise and $\tan^{-1}$ denotes the ordinary one-argument arctangent. This descriptor in \eqref{eq:phi_s} is used as a numerical feature and is not assigned the physical interpretation of phase in a latent feature field. We then construct $$  \mathbf S^{(s)}:
  =
  \operatorname{Concat}\left[
    \left|\bm{\mathcal R}^{(s)}_1\right|,
    \left|\bm{\mathcal R}^{(s)}_2\right|,
    \mathbf\Phi^{(s)}
    \right]
  \in
  \mathbb R^{B\times3C_s\times H_s\times W_s}.$$
A biased $1\times1$ projection $\psi_{\mathrm R}^{(s)}:
\mathbb R^{3C_s}\rightarrow\mathbb R^{C_s}$ 
is applied at each spatial location to fuse the descriptor channels. 

If there is no activation or normalization, then  \eqref{eq:delta_X_s} becomes the variant 
$
  \Delta\mathbf X^{(s)}
  =
  \psi_{\mathrm R}^{(s)}
  \left(\mathbf S^{(s)}\right).
$
Equation~\eqref{eq:outputY_s} then inserts this correction into the host feature path.

\subsection{Learnable LCT Reparameterization with Endpoint Conditioning}\label{FR-C}

The learnable LCT mode is deployed at selected latent insertion sites and does not use a directional multiplier. The learned matrix pair $\boldsymbol A^{(s)}$ is denoted by $\boldsymbol A_{\mathrm L}^{(s)}$ to distinguish it from the fixed Fourier pair. We now propose a module which consists of  the spatial endpoint response, the fixed Fourier response, and the response in learned linear canonical coordinates as follows: $ \mathbf U_{\mathrm S}^{(s)}:=
\phi_{\mathrm S}^{(s)}
(\mathbf X^{(s)}), 
$
$$
    \mathbf U_{\mathrm L}^{(s)}:
      =
    \rho_{\mathrm L}^{(s)}
    [
      \mathscr L_{\boldsymbol A_{\mathrm L}^{(s)}}
      (\mathbf X^{(s)})
    ],
                 \        {\rm and} \ 
    \mathbf U_{\mathrm F}^{(s)}:
      =
    \rho_{\mathrm F}^{(s)}
    [
      \mathscr L_{\boldsymbol A_{\mathrm F}}
      (\mathbf X^{(s)})
      ].$$
Here,
$\phi_{\mathrm S}^{(s)}$
denotes a depthwise $3\times3$ convolution with unit stride, dimension-preserving padding, and a bias term. Moreover,  $\mathbf U_{\mathrm S}^{(s)}$ is not followed by an additional activation or normalization, and $\boldsymbol A_{\mathrm L}^{(s)}$
uses the parameterization in
\eqref{eq:stage_canonical_parameterization}. The fixed Fourier endpoint is $\boldsymbol A_{\mathrm F}:
=
(A_{\mathrm F},A_{\mathrm F}),$
where 
 $A_{\mathrm F}:
  =
  \begin{pmatrix}
    0  & 1 \\
    -1 & 0
  \end{pmatrix},$
which is evaluated using an orthonormal two-dimensional FFT.

The Fourier and learnable LCT branches both use forward transform
coefficients, but apply different real-valued readouts. For
$\mathbf Z\in\mathbb C^{B\times C_s\times H_s\times W_s}$, define
\begin{equation}\label{eq:transform_readouts}
\begin{aligned}
\rho_{\mathrm F}^{(s)}(\mathbf Z)
=P_{\mathrm F}^{(s)}
  \!\left(\log(1+|\mathbf Z|)\right),\\
\rho_{\mathrm L}^{(s)}(\mathbf Z)=P_{\mathrm L}^{(s)}
  \left(
  [\operatorname{Re}\mathbf Z,
   \operatorname{Im}\mathbf Z]_{\mathrm{ch}}
  \right).
\end{aligned}
\end{equation}
Here, $[\cdot,\cdot]_{\mathrm{ch}}$ denotes concatenation along the
channel dimension. $P_{\mathrm F}^{(s)}$ is a learned biased
$1\times1$ projection from $C_s$ to $C_s$ channels, whereas
$P_{\mathrm L}^{(s)}$ is a learned biased $1\times1$ projection from
$2C_s$ to $C_s$ channels. No activation or normalization follows
either projection. Thus, $\mathbf U_{\mathrm S}^{(s)}$,
$\mathbf U_{\mathrm F}^{(s)}$, and $\mathbf U_{\mathrm L}^{(s)}$ are
real tensors in $\mathbb R^{B\times C_s\times H_s\times W_s}$.

These transform branches produce coefficient-domain descriptors,
without reconstructing spatial signals by an inverse transform.
The coefficient maps are treated as feature descriptors indexed on
a common tensor grid. The fixed Fourier branch uses coefficient
magnitudes as a prescribed spectral reference, whereas the learned
LCT branch projects the real and imaginary coefficient maps jointly
and can therefore respond to their phase. Endpoint conditioning
modulates the learned LCT response without changing the insertion
site or the module mode. For $\mathbf U_{\mathrm F}^{(s)}$ and
$\mathbf U_{\mathrm L}^{(s)}$, the readouts refer to transform
coefficients rather than physical spatial locations.

The responses in spatial and Fourier coordinates provide two endpoint references for conditioning the learned LCT branch. Before their energies are measured, we apply the same channel normalization without affine parameters for $\mathbf U_{\mathrm S}^{(s)}$ and
$\mathbf U_{\mathrm F}^{(s)}$. For a response $\mathbf U$, let $ \mu_{b,h,w}(\mathbf U):
  =
  \frac{1}{C_s}
  \sum_{c=1}^{C_s}
  U_{b,c,h,w},$
and define
$
  \nu_{b,h,w}(\mathbf U):
  =
  \frac{1}{C_s}
  \sum_{c'=1}^{C_s}
  \left[
  U_{b,c',h,w}-\mu_{b,h,w}(\mathbf U)
  \right]^2.
$
The normalized response is
$$
  \operatorname{Norm}(\mathbf U)_{b,c,h,w}:
  =
  \frac{
    U_{b,c,h,w}-\mu_{b,h,w}(\mathbf U)
  }{
    \sqrt{\nu_{b,h,w}(\mathbf U)+\varepsilon_{\mathrm N}}
  }.
$$
The normalized response energy in each channel is $ \mathbf e_m^{(s)}:
  =
  \operatorname{GAP}
  (
  |
  \operatorname{Norm}
  (\mathbf U_m^{(s)})
  |^2
  ),
  m\in\{\mathrm S,\mathrm F\},$
where GAP averages over the spatial dimensions, so
$\mathbf e_m^{(s)}\in\mathbb R^{B\times C_s}$.
We form the regularized relative energies $  \mathbf p_m^{(s)}:
  =
  \frac{
    \mathbf e_m^{(s)}
  }{
    \mathbf e_{\mathrm S}^{(s)}
    +
    \mathbf e_{\mathrm F}^{(s)}
    +
    \varepsilon_{\mathrm E}
  },$
and the statistic measuring the balance between the endpoint responses
$$
  \mathbf h_{\mathrm{SF}}^{(s)}:
  =
  -
  \frac{
    \sum\limits_{m\in\{\mathrm S,\mathrm F\}}
    \mathbf p_m^{(s)}
    \log
    (\mathbf p_m^{(s)}+\varepsilon_{\log})
  }{
    \log 2
  }.
$$
Because of the denominator stabilizer, the two relative energies do not sum exactly to one, although the deviation becomes negligible when their total energy is large relative to \(\varepsilon_{\mathrm E}\). Accordingly,
\(\mathbf h_{\mathrm{SF}}^{(s)}\)
is used for conditioning without assigning it an exact probability interpretation or a prescribed range. When the endpoint energy is well above the regularization scale, the statistic is small when one endpoint dominates, and larger when the two relative energies are comparable.

The learned-LCT branch weight is
$  \lambda_{\mathrm L}^{(s)}:
  =
  \sigma
  (
  \boldsymbol\alpha^{(s)}
  \odot
  \mathbf h_{\mathrm{SF}}^{(s)}
  +
  \boldsymbol\beta^{(s)}
  ),$
where $\sigma$ denotes the sigmoid function and $\boldsymbol\alpha^{(s)},\boldsymbol\beta^{(s)} \in\mathbb R^{C_s}$ are parameters learned separately for each channel. The resulting $\lambda_{\mathrm L}^{(s)} \in\mathbb R^{B\times C_s}$ is a broadcast over the spatial dimensions. Consequently, a positive, negative, or zero value of each channel of $\boldsymbol\alpha^{(s)}$ makes the channelwise mapping from $h_{\mathrm{SF},n,c}^{(s)}$ to
$\lambda_{\mathrm L,n,c}^{(s)}$ increasing, decreasing, or constant, respectively.

The residual correction is
\begin{align*}
  \Delta\mathbf X^{(s)}:
  =
  \psi_{\mathrm L}^{(s)}
  \left(
  \mathbf U_{\mathrm S}^{(s)}
  +
  \mathbf U_{\mathrm F}^{(s)}
  +
  \lambda_{\mathrm L}^{(s)}
  \odot
  \mathbf U_{\mathrm L}^{(s)}
  \right),
\end{align*}
where
\(\psi_{\mathrm L}^{(s)}\)
denotes a learned \(1\times1\) channel projection with bias and no subsequent activation or normalization. This correction is inserted through \eqref{eq:outputY_s}. The input-dependent coefficient \(\lambda_{\mathrm L}^{(s)}\) modulates only the learned LCT response, and it neither selects the insertion site nor switches between the two modes.

\subsection{Implementation and Computational Cost}\label{FR-D}

The discrete transforms are applied independently to each feature channel. Canonical matrices are shared across channels at one insertion site, and the response projections, fusion operators, and residual gates are learned for that site. Modules attached to distinct nonrecurrent sites do not share parameters. For a recurrent host block, the parameters of the LCT module are reused whenever the host block itself is reused.

 For the learnable LCT mode and the LCRT mode, we make the following initialization assumptions. 
 Let \(\kappa_k^{(s)}=\eta_k^{(s)}=0\) and
\(\theta_k^{(s)}=\pi/2\). Let \(\boldsymbol\alpha^{(s)}=\boldsymbol\beta^{(s)}=\mathbf0\)
and, for \(\gamma^{(s)}\), we retain the zero initialization. For convolution weights, we use Kaiming uniform initialization, and their biases are initialized to zero. We use
\(\varepsilon_{\mathrm R}=\varepsilon_{\Phi}
=\varepsilon_{\mathrm N}=\varepsilon_{\mathrm E}
=\varepsilon_{\log}=10^{-6}\)
at every insertion site.

A separable two-dimensional LCT is implemented by chirp multiplication and FFTs along both tensor axes. For $\mathbf X^{(s)}$, one transform requires $ \mathcal O(
  B C_s H_s W_s
  [\log H_s+\log W_s]
  ) $
operations. For approximately square feature maps with
$H_s\approx W_s\approx N_s$, the transform cost simplifies into
$
  \mathcal O(B C_sN_s^2\log N_s).
$
Each LCRT module evaluates one forward LCT and two inverse LCTs. The transform branches of the learnable LCT module evaluate one fixed Fourier transform and one learned forward LCT, without any inverse transform.

The depthwise spatial mapping and the endpoint statistics require $\mathcal O(B C_s H_s W_s)$ operations. The LCRT fusion projection maps $3C_s$ channels to $C_s$ channels. In the learnable LCT mode, the Fourier response projection maps $C_s$ to $C_s$ channels, the real--imaginary LCT response projection maps $2C_s$ to $C_s$ channels, and the fusion projection maps $C_s$ to $C_s$ channels. These projections require $\mathcal O(B C_s^2 H_s W_s)$ operations, with constants determined by their input and output channel dimensions. Practical overhead therefore depends on the feature resolution, channel width, and number of independent inserted modules.
Section~\ref{sv} reports parameters, FLOPs, and latency for the complete networks used in the experiments.

\begin{table*}[t] 
  \centering 
  \footnotesize 
  \caption{LCRT insertion locations for crack segmentation with 
    $448\times448$ inputs. Feature dimensions are 
    $C_s\times H_s\times W_s$.} 
  \label{tab:crack_module_placement} 
 
  \begin{tabularx}{\textwidth}{ 
    @{} 
    >{\raggedright\arraybackslash}p{0.14\textwidth} 
    >{\raggedright\arraybackslash}X 
    >{\raggedright\arraybackslash}p{0.28\textwidth} 
    @{} 
    } 
    \toprule 
    \textbf{Backbone}                                 & 
    \textbf{Insertion location}                       & 
    \textbf{Module count and feature dimensions}        \\ 
    \midrule 
 
    U-Net~\cite{Ronneberger2015UNet}                  & 
    Third encoder skip feature $e_3$, before concatenation 
    with the corresponding decoder feature            & 
    1 module; $256\times112\times112$                 \\ 
 
    U-Net v2~\cite{Peng2023UNetV2}                    & 
    Projected shallow feature $f_1$ after 
    \texttt{Translayer\_1} and before SDI aggregation & 
    1 module; $32\times112\times112$                  \\ 
 
    Attention U-Net~\cite{Schlemper2019AttentionGatedNetworks}     & 
    Third encoder skip feature $e_3$, before its attention 
    gate and subsequent decoder concatenation         & 
    1 module; $256\times112\times112$                 \\ 
 
    DeepLabV3+~\cite{Chen2018DeepLabV3Plus}           & 
    Projected low-level feature after the $1\times1$ 
    projection, normalization, and activation, before 
    concatenation with the upsampled ASPP feature     & 
    1 module; $48\times112\times112$                  \\ 
 
    \bottomrule 
  \end{tabularx}

  \vspace{0.5mm}

  \parbox{\textwidth}{\scriptsize
    U-Net and Attention U-Net use encoder widths
    $64,128,256,512,1024$, with $e_3$ denoting the feature after two downsampling operations. U-Net v2 uses a PVTv2-B2 encoder and a 32-channel projected feature before SDI fusion. DeepLabV3+ uses an encoder output stride of 16 and a stride-4 low-level branch projected to 48 channels. LCRT preserves the feature dimensions at each insertion site.
  }
\end{table*}

\section{Experimental Evaluation}\label{sv}

This section evaluates the task performance and computational cost of the two modes of LLCT. We first present the experimental setup in Subsection \ref{EE-A}, and then evaluate LCRT on crack segmentation in Subsection \ref{EE-B} and learnable LCT on single-image deraining through backbone comparisons and component
ablations in Subsection \ref{EE-C}. We further report parameter counts, computational cost, and inference latency under matched settings in Subsection \ref{EE-D}.

\subsection{Experimental Setup}\label{EE-A}

\textbf{Tasks, Baselines, and Metrics.}
The main evaluation considers crack segmentation and single-image deraining. LCRT is assessed on crack segmentation, and the learnable LCT is assessed on deraining. For each experiment reported in the main text, the mode and insertion site are specified in advance and remain fixed during training and inference, as described in Section~\ref{siv}.

Crack segmentation is evaluated on the Crack Segmentation Dataset~\cite{KaggleCrackSegmentationDataset} using U-Net~\cite{Ronneberger2015UNet}, U-Net v2~\cite{Peng2023UNetV2}, Attention U-Net~\cite{Schlemper2019AttentionGatedNetworks}, and DeepLabV3+~\cite{Chen2018DeepLabV3Plus}. The deraining evaluation uses Rain200L, Rain200H, DID, and SPA with DDN~\cite{Fu2017DDN}, PReNet~\cite{Ren2019PReNet}, Restormer~\cite{Zamir2022Restormer}, and DRSformer~\cite{Chen2023DRSformer}. We report precision, recall, F1 score, and IoU for the crack class, and PSNR and SSIM for the deraining task.

\textbf{Comparison Protocol and Module Placement.}
All models are implemented in PyTorch and evaluated on a workstation equipped with two NVIDIA RTX A6000 GPUs. For each architecture, the
unmodified host network serves as the baseline. The baseline and its augmented counterpart use the same task-specific training and evaluation protocol. The inserted module is optimized jointly with the host network, while the prediction head, training objective, and input and output formats remain unchanged.

We insert one LCRT module into a feature branch at one-quarter of the input resolution in each backbone. Table~\ref{tab:crack_module_placement} specifies the LCRT insertion
locations used in the crack-segmentation experiments. For deraining, Table~\ref{tab:learnable_lct_deraining_results} reports the insertion site and scope for each backbone. The canonical initialization, numerical constants, and rules for sharing parameters follow Section~\ref{siv}, and in particular, modules at distinct nonrecurrent sites use separate parameters, and an insertion in a recurrent block follows the sharing pattern of that block.

\begin{figure}[t]
  \centering
  \includegraphics[width=\linewidth]{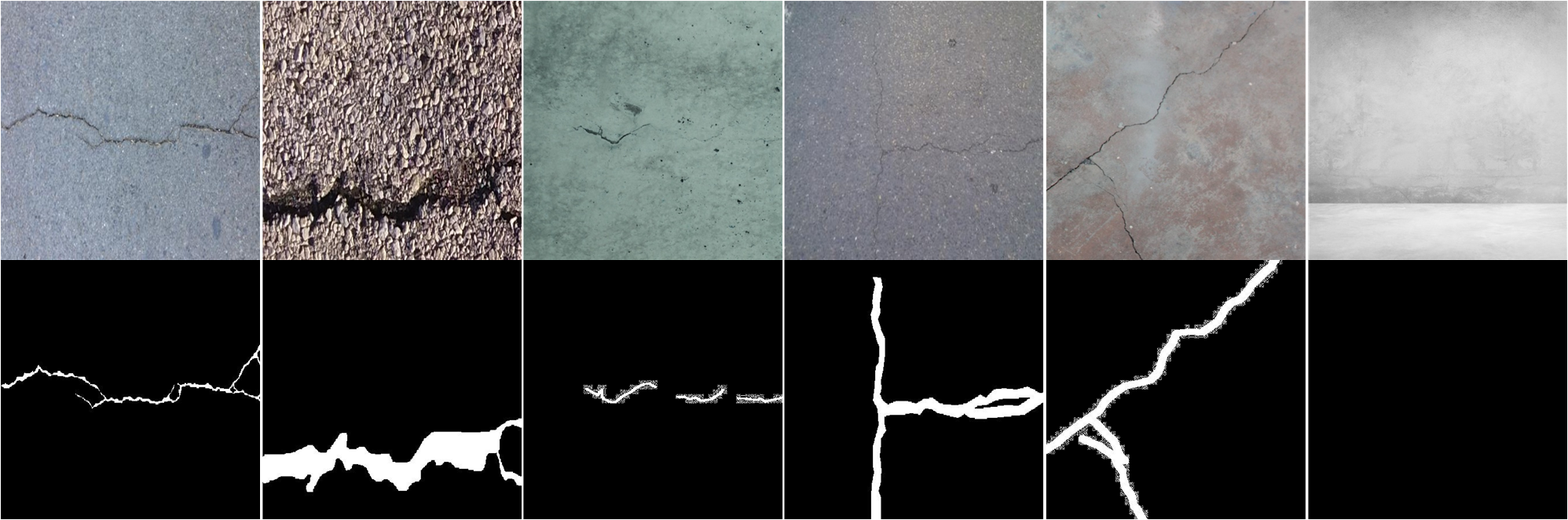}
  \caption{Example image--mask pairs from the Crack Segmentation Dataset. Top row: RGB images. Bottom row: ground-truth masks, with crack pixels shown in white and background pixels in black.}
  \label{fig:crackdata}
\end{figure}

\begin{table}[t]
  \renewcommand{\arraystretch}{1}
  \caption{LCRT component ablation results on the Crack Segmentation Dataset. All metrics are reported for the crack class.}
  \label{tab:lcrt_crack_ablation}
  \centering
  \scriptsize
  \begin{tabular*}{\columnwidth}{@{\extracolsep{\fill}}lcccc@{}}
    \toprule
    \textbf{Method}
      & \textbf{Prec. $\uparrow$}
      & \textbf{Rec. $\uparrow$}
      & \textbf{F1 $\uparrow$}
      & \textbf{IoU $\uparrow$} \\
    \midrule
    \multicolumn{5}{@{}l}{\textit{U-Net}} \\
    Baseline     & 0.7239          & 0.7199          & 0.7219          & 0.5648          \\
    +Riesz-only  & 0.7278          & 0.7350          & 0.7314          & 0.5765          \\
    +forward-LCT-only    & 0.7358          & 0.7285          & 0.7321          & 0.5775          \\
    +LCRT (Ours) & \textbf{0.7370} & \textbf{0.7402} & \textbf{0.7386} & \textbf{0.5855} \\
    \midrule
    \multicolumn{5}{@{}l}{\textit{U-Net v2}} \\
    Baseline     & 0.7044          & 0.7424          & 0.7229          & 0.5660          \\
    +Riesz-only  & 0.7165          & \textbf{0.7501} & 0.7329          & 0.5784          \\
    +forward-LCT-only    & 0.7248          & 0.7430          & 0.7338          & 0.5795          \\
    +LCRT (Ours) & \textbf{0.7289} & 0.7492          & \textbf{0.7389} & \textbf{0.5859} \\
    \midrule
    \multicolumn{5}{@{}l}{\textit{Att. U-Net}} \\
    Baseline     & 0.7101          & 0.7330          & 0.7213          & 0.5640          \\
    +Riesz-only  & 0.7188          & 0.7425          & 0.7305          & 0.5754          \\
    +forward-LCT-only    & 0.7241          & 0.7368          & 0.7304          & 0.5753          \\
    +LCRT (Ours) & \textbf{0.7279} & \textbf{0.7464} & \textbf{0.7370} & \textbf{0.5836} \\
    \midrule
    \multicolumn{5}{@{}l}{\textit{DeepLabV3+}} \\
    Baseline     & 0.7262          & 0.7101          & 0.7181          & 0.5602          \\
    +Riesz-only  & 0.7255          & 0.7280          & 0.7267          & 0.5708          \\
    +LCT-only    & 0.7353          & 0.7195          & 0.7273          & 0.5715          \\
    +LCRT (Ours) & \textbf{0.7361} & \textbf{0.7334} & \textbf{0.7347} & \textbf{0.5807} \\
    \bottomrule
  \end{tabular*}
\end{table}

\begin{figure*}[t]
  \centering
  \includegraphics[width=0.98\textwidth]{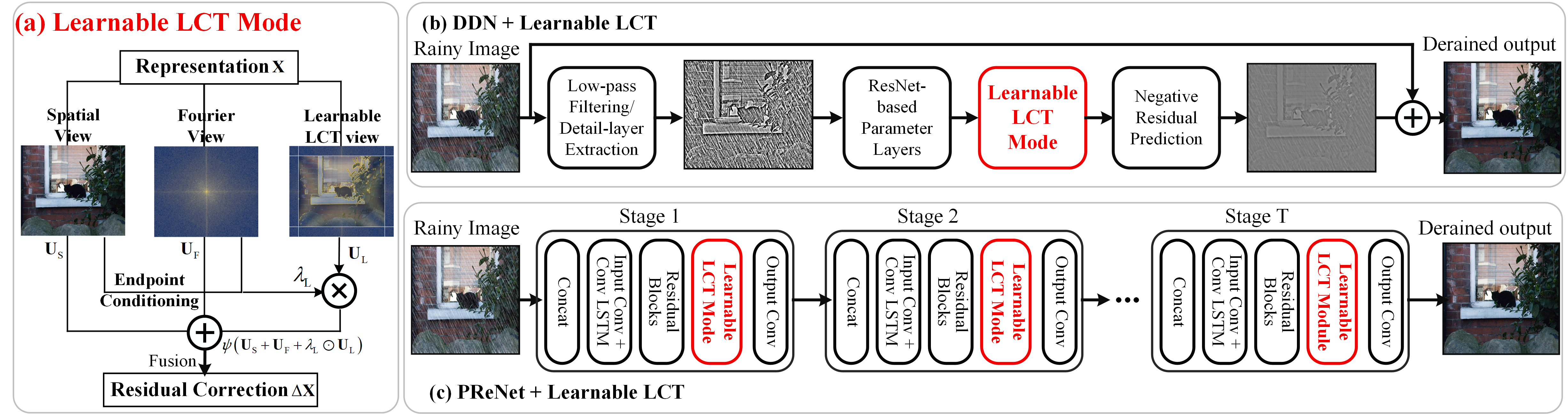}
  \caption{Learnable LCT module and its insertion into image deraining backbones. (a) The module combines views in spatial, Fourier, and learned canonical coordinates, and the endpoint activity determines the weight of the learned view. (b) In DDN, the module is inserted after the ResNet-based parameter layers and before negative residual prediction. (c) In PReNet, it is inserted within each recurrent stage between the residual blocks and the output convolution.}
  \label{fig:learnable_lct_derain_insertion}
\end{figure*}

\begin{table*}[h]
  \centering
  \footnotesize
  \renewcommand{\arraystretch}{1} \setlength{\tabcolsep}{4pt}
  \caption{Training configurations for the image deraining experiments.}
  \label{tab:deraining_training}
  \begin{tabularx}{\textwidth}{
    @{}
    >{\raggedright\arraybackslash}p{0.105\textwidth}
    >{\raggedright\arraybackslash}p{0.12\textwidth}
    >{\raggedright\arraybackslash}p{0.16\textwidth}
    >{\raggedright\arraybackslash}X
    >{\centering\arraybackslash}p{0.10\textwidth}
    >{\centering\arraybackslash}p{0.095\textwidth}
    >{\centering\arraybackslash}p{0.07\textwidth}
    @{}
  }
    \toprule
    \textbf{Backbone} &
    \textbf{Loss} &
    \textbf{Optimizer / initial learning rate} &
    \textbf{Learning-rate schedule} &
    \textbf{Training iterations} &
    \textbf{Patch size} &
    \textbf{Global batch size} \\
    \midrule
    DDN &
    Squared reconstruction error &
    Adam / $10^{-3}$ &
    Multiplied by 0.1 at 100K and 150K iterations &
    200K & $64\times64$ & 20 \\

    PReNet &
    $-\mathrm{SSIM}$ &
    Adam / $10^{-3}$ &
    Multiplied by 0.2 at 60K, 100K, and 160K iterations &
    200K & $100\times100$ & 18 \\

    Restormer &
    $L_1$ &
    AdamW / $3\times10^{-4}$ &
    Constant for 92K iterations, then cosine decay to
    $10^{-6}$ over 208K iterations &
    300K & $128\times128$ & 8 \\

    DRSformer &
    $L_1$ &
    AdamW / $10^{-4}$ &
    Constant for 92K iterations, then cosine decay to
    $10^{-6}$ over 208K iterations &
    300K & $128\times128$ & 8 \\
    \bottomrule
  \end{tabularx}

  \vspace{0.5mm}
  \parbox{\textwidth}{\scriptsize
    K denotes \(10^3\) optimizer updates. Global batch size denotes the
    number of training samples contributing to one optimizer update.
    Patch sizes remain fixed throughout training.
  }
\end{table*}

\subsection{Directional Reparameterization via LCRT}\label{EE-B}

\textbf{Crack Segmentation.}
We evaluate the LCRT on the Kaggle Crack Segmentation
Dataset~\cite{KaggleCrackSegmentationDataset}, which contains
11,298 image--mask pairs resized to $448\times448$.
The dataset aggregates samples from 12 public
crack segmentation datasets and includes both crack-containing
and crack-free images. We retain the released stratified test
partition of 1,695 images and split the remaining 9,603 training
images into training and validation subsets using a fixed random
seed. The validation subset comprises approximately 10\% of the
released training partition, with stratification according to the
source-dataset identifiers encoded in the image filenames.

All models are trained for 50 epochs with a batch size of 4 using Adam and an initial learning rate of $10^{-3}$. Binary cross-entropy with logits is used as the training objective. Input images are converted to RGB, resized to $448\times448$, and normalized channelwise using a mean and standard deviation of 0.5. The corresponding masks are resized using nearest-neighbor interpolation and binarized at 0.5.

For each configuration, we retain the checkpoint with the highest validation-set Dice score at a threshold of 0.5. For that checkpoint, a single decision threshold is selected by maximizing validation Dice over $\{0.30,0.35,\ldots,0.70\}$ and then fixed for test evaluation. Crack class precision, recall, F1 score, and IoU are computed from foreground-pixel counts aggregated over the complete test set.
Table~\ref{tab:lcrt_crack_ablation} compares the baseline, Riesz-only, LCT-only, and full LCRT configurations. Fig.~\ref{fig:crackdata} shows six representative image--mask pairs from the dataset.

The baseline denotes the original backbone without the inserted module. All inserted-module variants use the same insertion location, feature dimensions, projection layers, residual gate, and training protocol. The Riesz-only variant applies the regularized Riesz multiplier in fixed Fourier coordinates without learning the canonical coordinates. The LCT-only variant applies the learned forward LCT to the input features and converts its coefficients into a real-valued response. It uses neither the Riesz multiplier nor an inverse LCT. Full LCRT combines the learnable canonical coordinates with the regularized Riesz response as described in Section~\ref{siv}.

\textbf{Quantitative Results.}
For the crack class, we have the following quantitative results.  The full LCRT point estimates are {0.0157--0.0167} higher in F1 and {0.0196--0.0207} higher in IoU than the corresponding baseline values. Compared with the better value from either of individual components for each backbone, the full configuration is higher by 0.0051--0.0074 in F1 and 0.0064--0.0092 in IoU.
For each backbone, the Riesz-only has higher recall than the LCT-only, while the LCT-only has higher precision. Among the reported point estimates, the full LCRT has the highest F1 and IoU in every backbone group and the highest precision in all four groups. It has the highest recall in three groups, and for U-Net v2, the Riesz-only recall is 0.0009 higher.

\begin{table*}[t]
  \renewcommand{\arraystretch}{1}
  \caption{Learnable LCT insertion settings and paired image deraining results for baseline models
    and the corresponding models augmented with learnable LCT on Rain200L, Rain200H, DID, and SPA.}
  \label{tab:learnable_lct_deraining_results}
  \centering
  \begin{threeparttable}
    \scriptsize
    \setlength{\tabcolsep}{2.6pt}
    \resizebox{\textwidth}{!}{%
      \begin{tabular}{lcc*{8}{c}}
        \toprule
        \multirow{2}{*}{Method}
                       & \multirow[c]{2}{*}{\shortstack[c]{Insertion Site}}
                       & \multirow[c]{2}{*}{\shortstack[c]{Insertion Scope}}
                       & \multicolumn{2}{c}{Rain200L}
                       & \multicolumn{2}{c}{Rain200H}
                       & \multicolumn{2}{c}{DID}
                       & \multicolumn{2}{c}{SPA}                                                                                                                                                                                                                                        \\
        \cmidrule(lr){4-5}\cmidrule(lr){6-7}
        \cmidrule(lr){8-9}\cmidrule(lr){10-11}
                       &                                                          &                 & PSNR (dB) \(\uparrow\) & SSIM \(\uparrow\) & PSNR (dB) \(\uparrow\) & SSIM \(\uparrow\) & PSNR (dB) \(\uparrow\) & SSIM \(\uparrow\) & PSNR (dB) \(\uparrow\) & SSIM \(\uparrow\) \\
        \midrule

        DDN \textit{(Baseline)}
                       & \multirow[c]{2}{*}{\shortstack[c]{After parameter layers                                                                                                                                                                                                       \\based on ResNet}}
                       & \multirow[c]{2}{*}{\shortstack[c]{Single                                                                                                                                                                                                                       \\insertion}}
                       & 34.39                                                    & 0.9626          & 26.30                  & 0.8067
                       & 31.28                                                    & 0.9184          & 36.83                  & 0.9446                                                                                                                                                   \\
        \cmidrule(lr){1-1}\cmidrule(lr){4-11}
        \hspace{0.8em}$+$ learnable LCT
                       &                                                          &                 &
        \textbf{34.92} & \textbf{0.9628}
                       & \textbf{26.83}                                           & \textbf{0.8232}
                       & \textbf{32.10}                                           & \textbf{0.9202}
                       & \textbf{36.91}                                           & \textbf{0.9489}                                                                                                                                                                                     \\

        \midrule

        PReNet \textit{(Baseline)}
                       & \multirow[c]{2}{*}{\shortstack[c]{After residual                                                                                                                                                                                                               \\blocks}}
                       & \multirow[c]{2}{*}{\shortstack[c]{Each recurrent                                                                                                                                                                                                               \\stage}}
                       & 37.21                                                    & 0.9779          & 28.94                  & 0.8958
                       & 33.31                                                    & \textbf{0.9418} & 39.69                  & 0.9840                                                                                                                                                   \\
        \cmidrule(lr){1-1}\cmidrule(lr){4-11}
        \hspace{0.8em}$+$ learnable LCT
                       &                                                          &                 &
        \textbf{38.65} & \textbf{0.9873}
                       & \textbf{30.06}                                           & \textbf{0.9155}
                       & \textbf{33.66}                                           & 0.9402
                       & \textbf{40.86}                                           & \textbf{0.9879}                                                                                                                                                                                     \\

        \midrule

        Restormer \textit{(Baseline)}
                       & \multirow[c]{2}{*}{\shortstack[c]{After MDTA and                                                                                                                                                                                                               \\before GDFN}}
                       & \multirow[c]{2}{*}{\shortstack[c]{Bottleneck/latent                                                                                                                                                                                                            \\Transformer blocks}}
                       & 40.95                                                    & 0.9889          & 31.96                  & 0.9324
                       & 35.32                                                    & 0.9643          & 47.98                  & 0.9921                                                                                                                                                   \\
        \cmidrule(lr){1-1}\cmidrule(lr){4-11}
        \hspace{0.8em}$+$ learnable LCT
                       &                                                          &                 &
        \textbf{41.05} & \textbf{0.9893}
                       & \textbf{32.00}                                           & \textbf{0.9330}
                       & \textbf{35.79}                                           & \textbf{0.9651}
                       & \textbf{48.16}                                           & \textbf{0.9922}                                                                                                                                                                                     \\

        \midrule

        DRSformer \textit{(Baseline)}
                       & \multirow[c]{2}{*}{\shortstack[c]{After TKSA and                                                                                                                                                                                                               \\before MSFN}}
                       & \multirow[c]{2}{*}{\shortstack[c]{Bottleneck/latent                                                                                                                                                                                                            \\STBs}}
                       & 41.20                                                    & 0.9889          & 32.13                  & 0.9324
                       & 35.30                                                    & 0.9638          & \textbf{48.48}         & 0.9919                                                                                                                                                   \\
        \cmidrule(lr){1-1}\cmidrule(lr){4-11}
        \hspace{0.8em}$+$ learnable LCT
                       &                                                          &                 &
        \textbf{41.25} & \textbf{0.9898}
                       & \textbf{32.31}                                           & \textbf{0.9369}
                       & \textbf{35.69}                                           & \textbf{0.9667}
                       & 48.43                                                    & \textbf{0.9922}                                                                                                                                                                                     \\

        \bottomrule
      \end{tabular}
    }
  \end{threeparttable}
\end{table*}

\subsection{Learnable LCT for Image Deraining}\label{EE-C}

\textbf{Cross-Backbone Deraining Results.}
We evaluate image deraining on Rain200L and Rain200H~\cite{Yang2017DeepJointRain}, DID-Data~\cite{Zhang2018DIDMDN}, and SPA-Data~\cite{Wang2019SPANet}. Rain200L and Rain200H each contain 1,800 training and 200 test pairs, while DID-Data contains 12,000 training and 1,200 test pairs. For SPA-Data, we use the benchmark version containing 638,492 training and 1,000 test pairs.

Training follows the backbone-specific configurations in Table~\ref{tab:deraining_training}. In DDN, the module is inserted after the ResNet-based parameter layers and before negative residual prediction. In PReNet, it is inserted after the residual
blocks and before the output convolution within each recurrent stage. In the Transformer backbones, it is inserted between MDTA
and GDFN in the Restormer bottleneck and between TKSA and MSFN in the DRSformer bottleneck, as illustrated in Fig.~\ref{fig:learnable_lct_derain_insertion}.

Table~\ref{tab:learnable_lct_deraining_results} reports higher PSNR values for the learnable LCT in 15 of the 16 backbone--dataset pairs and higher SSIM values in 15 of the 16 pairs. The mean baseline-relative PSNR gains across the four datasets are 0.49~dB for DDN, 1.02~dB for PReNet, 0.20~dB for Restormer, and 0.14~dB for DRSformer. Among the evaluated backbones, PReNet has the largest mean PSNR gain, including gains of 1.44~dB on Rain200L and 1.12~dB on Rain200H. Positive mean PSNR differences are observed for both the CNN and Transformer configurations, with larger values for the two evaluated CNN backbones.

Restormer has higher values for both metrics (PSNR and SSIM) on all four datasets. Two comparisons exhibit mixed metric changes: PReNet on DID-Data gains 0.35~dB in PSNR while SSIM decreases by 0.0016, and DRSformer on SPA-Data gains 0.0003 in SSIM while PSNR decreases by 0.05~dB. The observed effect therefore depends on the backbone, dataset, and metric.
Fig.~\ref{fig:derain_qualitative} shows qualitative results for DDN and PReNet. In the selected PReNet examples in panel~(b), the outputs obtained with the learnable LCT show fewer visible rain streaks while retaining local boundaries in the displayed regions.

\textbf{Ablation Studies.}
Table~\ref{tab:learnable_lct_deraining_ablation} compares five PReNet configurations on Rain200L and Rain200H. The S+F variant retains the spatial and Fourier responses but removes the LCT branch. The fixed-matrix variant restores the LCT branch and its real--imaginary coefficient projection while setting $\boldsymbol A_{\mathrm L}=\boldsymbol A_{\mathrm F}$. Both the fixed-matrix and learned-matrix/static-gate variants set $\boldsymbol\alpha=\mathbf0$ and learn $\boldsymbol\beta$, so the LCT-branch weight does not depend on the input. The latter variant learns $\boldsymbol A_{\mathrm L}$. The full model learns $\boldsymbol A_{\mathrm L}$ and uses the input-dependent weight defined in Section~\ref{FR-C}. 

Relative to the baseline, S+F increases PSNR by 0.07 and 0.24~dB on Rain200L and Rain200H, respectively, and its SSIM decreases by 0.0034 on Rain200L and increases by 0.0056 on Rain200H. Adding the fixed-matrix LCT branch to S+F gives PSNR differences of 0.48 and 0.29~dB and SSIM differences of 0.0016 and 0.0054. Under the same static-gate setting, learning the canonical matrix instead of fixing it at the Fourier configuration gives further PSNR differences of 0.36 and 0.38~dB and SSIM differences of 0.0061 and 0.0043. Relative to the learned-matrix/static-gate variant, the input-conditioned full model is higher by 0.53 and 0.21~dB in PSNR and by 0.0051 and 0.0044 in SSIM. The full model has the highest reported point estimates for both metrics on both datasets.

\begin{figure*}[t]
  \centering
  \begin{minipage}[t]{0.52\textwidth}
    \centering
    \includegraphics[width=\linewidth]{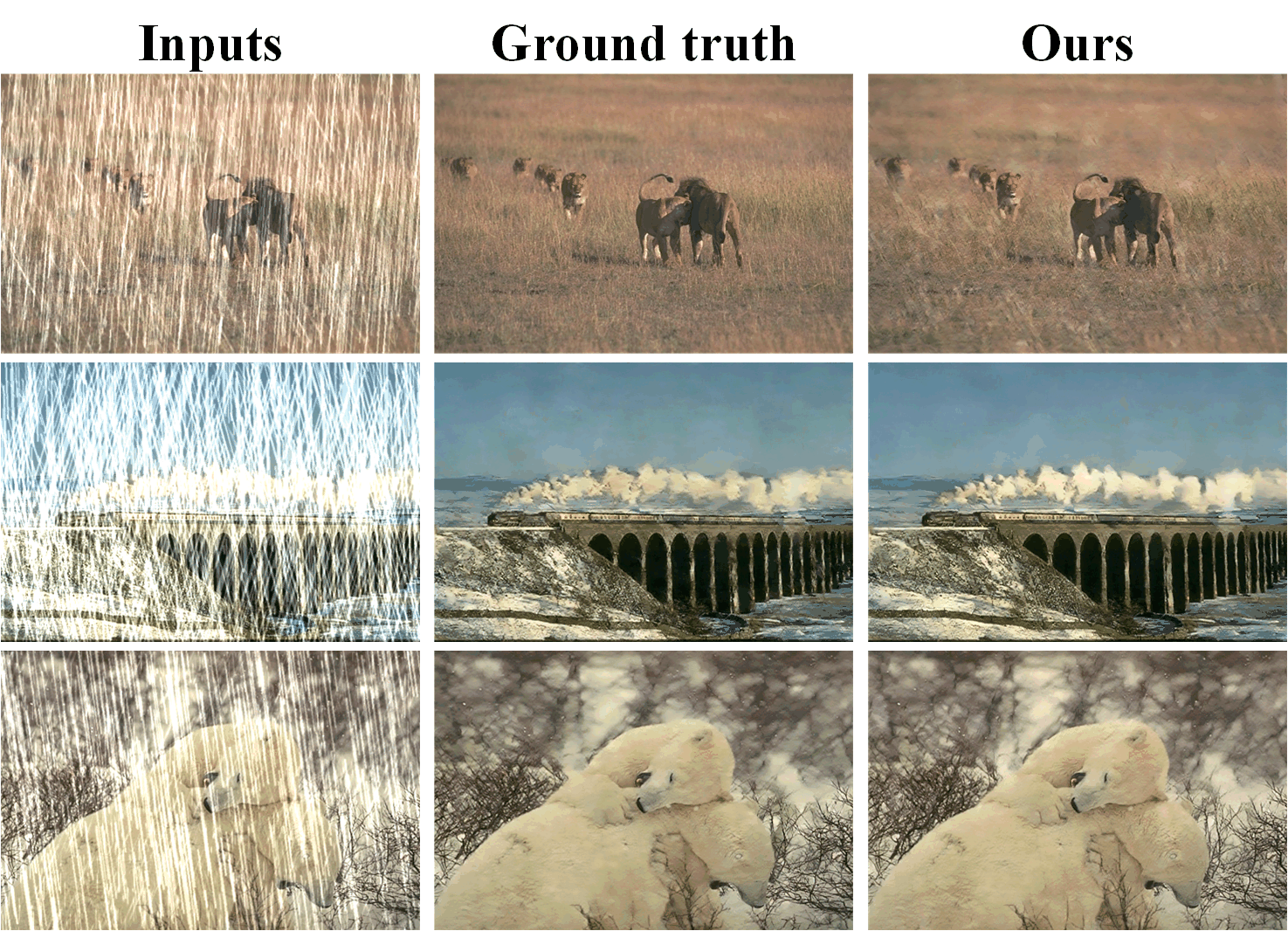}\\[-0.5mm]
    \footnotesize (a) DDN with learnable LCT
  \end{minipage}
  \hfill
  \begin{minipage}[t]{0.45\textwidth}
    \centering
    \includegraphics[width=\linewidth]{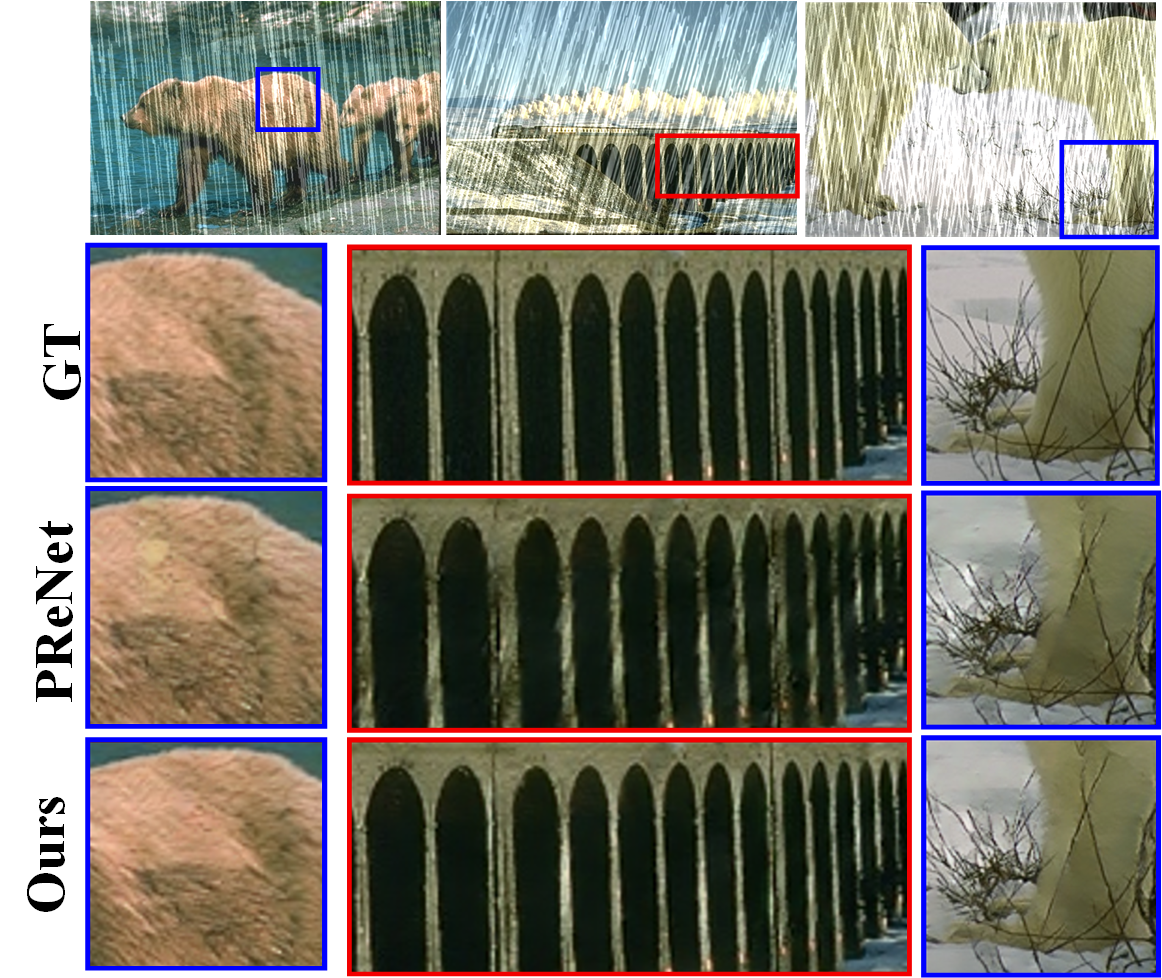}\\[-0.5mm]
    \footnotesize (b) PReNet comparison
  \end{minipage}
  \caption{Qualitative deraining results. (a) Rainy inputs, ground-truth images, and outputs from DDN with learnable LCT. (b) Rainy inputs and selected regions, followed by ground truth, PReNet outputs, and outputs from PReNet with learnable LCT.}
  \label{fig:derain_qualitative}
\end{figure*}

\begin{table}[H]
  \caption{Ablation of the LCT branch, canonical-matrix learning,
  and endpoint conditioning in PReNet.}
  \label{tab:learnable_lct_deraining_ablation}
  \centering
  \scriptsize
  \setlength{\tabcolsep}{3pt}
  \renewcommand{\arraystretch}{1.05}
  \begin{tabular*}{\columnwidth}{@{\extracolsep{\fill}}lcccc@{}}
    \toprule
    & \multicolumn{2}{c}{Rain200L}
    & \multicolumn{2}{c}{Rain200H} \\
    \cmidrule(lr){2-3}\cmidrule(lr){4-5}
    \textbf{Configuration}
    & \textbf{PSNR(dB)} & \textbf{SSIM}
    & \textbf{PSNR(dB)} & \textbf{SSIM} \\
    \midrule
    PReNet baseline
      & 37.21 & 0.9779 & 28.94 & 0.8958 \\
    S+F (no L)
      & 37.28 & 0.9745
      & 29.18 & 0.9014 \\
    Fixed $\boldsymbol A_{\mathrm F}$
      & 37.76 & 0.9761
      & 29.47 & 0.9068 \\
    Learned $\boldsymbol A$, static gate
      & 38.12 & 0.9822
      & 29.85 & 0.9111 \\
    Full model
      & \textbf{38.65} & \textbf{0.9873} & \textbf{30.06} & \textbf{0.9155} \\
    \bottomrule
  \end{tabular*}
  \par\vspace{0.5mm}
  \parbox{\columnwidth}{\scriptsize
  S, F, and L denote the spatial, Fourier, and forward-LCT branches.
  The fixed-matrix variant includes all three branches and sets
  $\boldsymbol A_{\mathrm L}=\boldsymbol A_{\mathrm F}$ while retaining
  the real--imaginary projection. The static-gate variant learns
  $\boldsymbol A_{\mathrm L}$ but fixes $\boldsymbol\alpha=\mathbf 0$
  and learns $\boldsymbol\beta$.}
\end{table}

\subsection{Computational Efficiency}\label{EE-D}

Table~\ref{tab:empirical_profiling} reports parameter counts, profiler FLOPs, and measured latency for the selected configurations. Within each pair, the input resolution and inference setting are unchanged, and the only architectural difference is the insertion of the specified module.

For U-Net, LCRT adds 0.20~M parameters (0.6\%), 5.01~G profiler FLOPs (9.8\%), and 2.11~ms latency (17.2\%). For PReNet, the learnable LCT module adds 2.73\% parameters, 1.02~G profiler FLOPs (5.0\%), and 1.11~ms latency (2.4\%).

\begin{table}[H]
  \renewcommand{\arraystretch}{1.1}
  \caption{Parameters, profiler FLOPs, and measured latency for selected backbones with and without the corresponding module.}
  \label{tab:empirical_profiling}
  \centering
  \scriptsize
  \begin{tabular*}{\linewidth}{@{\extracolsep{\fill}}l ccc}
    \toprule
    \textbf{Model} & \textbf{Params (M)} & \textbf{FLOPs (G)} & \textbf{Latency (ms/sample)} \\
    \midrule
    \multicolumn{4}{@{}l}{\textit{Crack segmentation: $3\times448\times448$}} \\
    U-Net                  & 32.34 & 50.89 & 12.24 \\
    U-Net + LCRT           & 32.54 & 55.90 & 14.35 \\
    \midrule
    \multicolumn{4}{@{}l}{\textit{Deraining: $3\times100\times100$}} \\
    PReNet                 & 0.169 & 20.22 & 45.80 \\
    PReNet + learnable LCT & 0.174 & 21.24 & 46.91 \\
    \bottomrule
  \end{tabular*}
\end{table}

PReNet has a 2.73\% parameter increase and a 2.4\% latency increase, and U-Net has a 0.6\% parameter increase and a 17.2\% latency increase. The U-Net profile accompanies differences between point estimates of 0.0167 in F1 and 0.0207 in IoU for the crack class, while the PReNet profile accompanies a mean PSNR difference of 1.02~dB over the four deraining datasets. The reported performance and cost changes are specific to these configurations and are not normalized measures of cost effectiveness. Absolute latency and FLOPs should not be compared across the two tasks because their input sizes and host networks differ.

\section{Conclusion}

We presented LLCT, a framework that learns canonical parameters within the LCT family for visual feature processing at selected network sites. Each site uses canonical matrices with determinant one, parameterized by rotation, scaling, and shearing, together with a gated residual path initialized to zero. The LCRT mode combines learned canonical coordinates with a regularized Riesz-type directional response. The learnable LCT mode fuses responses from spatial, Fourier, and learned linear canonical coordinates, and the endpoint responses determine the weight of the learned branch. The two above modes assign different computational roles to a common coordinate family without redesigning the host prediction heads or task objectives. For the continuous transform, the accompanying theorem establishes the parameter stability for inputs under admissible nondegenerate perturbations of the canonical matrices. The discrete modules use the regularized chirp multiplication and the FFT on finite grids.

On the four evaluated crack-segmentation backbones, the full LCRT configuration has higher reported F1 and crack-class IoU than the corresponding baselines and its variants. The learnable LCT configuration has higher values in 30 of the 32 paired PSNR and SSIM measurements across four deraining backbones and four datasets, with a maximum PSNR difference of 1.44~dB. On the ablation studies, learning the canonical matrix gives higher PSNR and SSIM than fixing it at the Fourier configuration under static gating, while the input-conditioned full model has the highest reported values among the five configurations. These results support the tested configurations of the two modes across the evaluated tasks and backbones. The two profiling cases reported in the main text show latency increases of 2.4--17.2\% for the selected configurations.

Future work will evaluate the framework across a broader range of tasks and architectures, and examine its behavior under diverse data and computational settings. Further theoretical analysis will investigate how learned canonical coordinates influence feature representation, model generalization, and computational efficiency.

\appendices

\section{Proof of LCT Parameter Stability}\label{aa}
\begin{proof} [Proof of Theorem \ref{thm:lct-parameter-stability}]
Throughout this proof, we use the principal branch for all complex square roots. 

For convenience, write
$\boldsymbol{x}:=(x_1,x_2),$
$\boldsymbol{u}:=(u_1,u_2),$ and
$\boldsymbol{b}:=(b_1,b_2),$
and let
$\phi_{\boldsymbol{A}}(\boldsymbol{x}):=
\frac{1}{2}\sum_{k=1}^{2}\frac{a_k}{b_k}x_k^2$ and 
$\psi_{\boldsymbol{A}}(\boldsymbol{u})
:=\frac{1}{2}\sum_{k=1}^{2}\frac{d_k}{b_k}u_k^2.$
For each $m\in\mathbb N$, define
$\phi_{\boldsymbol{A}^{(m)}}(\boldsymbol{x})
:=\frac{1}{2}\sum_{k=1}^{2}
\frac{a_k^{(m)}}{b_k^{(m)}}x_k^2$ and
$\psi_{\boldsymbol{A}^{(m)}}(\boldsymbol{u})
:=\frac{1}{2}\sum_{k=1}^{2}
\frac{d_k^{(m)}}{b_k^{(m)}}u_k^2.$
For $\boldsymbol{b}\in(\mathbb R\setminus\{0\})^2$, let
$ D_{\boldsymbol{b}}h (\boldsymbol{u})
:=|b_1b_2|^{-1/2}h(\frac{u_1}{b_1},\frac{u_2}{b_2}).$
The definitions of the Fourier transform and the LCT kernel give
\begin{equation}
\label{eq:stability-lct-decomposition}
\mathscr{L}_{\boldsymbol{A}}(f)
=\chi(\boldsymbol{b})e^{i\psi_{\boldsymbol{A}}}
D_{\boldsymbol{b}}\mathscr{F}\bigl(e^{i\phi_{\boldsymbol{A}}}f\bigr),
\end{equation}
where
$\chi(\boldsymbol{b}):=\prod_{k=1}^{2}
\frac{|b_k|^{1/2}}{(ib_k)^{1/2}}$
and $|\chi(\boldsymbol{b})|=1$. With the unitary Fourier convention
of the main article, the normalization in
\eqref{eq:stability-lct-decomposition} follows from
$(2\pi)\prod_{k=1}^{2}(i2\pi b_k)^{-1/2}
=\chi(\boldsymbol{b})|b_1b_2|^{-1/2}.$
A change of variables also yields
\begin{equation}\label{eq:stability-dilation-isometry}
\|D_{\boldsymbol{b}}h\|_{L^2(\mathbb R^2)}
=\|h\|_{L^2(\mathbb R^2)}.
\end{equation}
 Define
$g_m:=e^{i\phi_{\boldsymbol{A}^{(m)}}}f$ and $g:=e^{i\phi_{\boldsymbol{A}}}f.$
Since
$\frac{a_k^{(m)}}{b_k^{(m)}}
\longrightarrow\frac{a_k}{b_k},\
k\in\{1,2\},$
it follows that $g_m(\boldsymbol{x})\to g(\boldsymbol{x})$ as $m\to\infty$ for every
$\boldsymbol{x}\in\mathbb R^2$. In addition,
$|g_m(\boldsymbol{x})-g(\boldsymbol{x})|^2
\leq 4|f(\boldsymbol{x})|^2.$
Thus, by the Lebesgue dominated convergence theorem, we obtain
\begin{equation}
\label{eq:stability-input-chirp-convergence}
\lim_{m\to\infty}\|g_m-g\|_{L^2(\mathbb R^2)}=0.
\end{equation}
\par\noindent Define
$h_m:=\mathscr{F}(g_m)$ and
$h:=\mathscr{F}(g).$
By Plancherel's theorem and
\eqref{eq:stability-input-chirp-convergence},
\begin{align}
\lim_{m\to\infty}
\|h_m-h\|_{L^2(\mathbb R^2)}
&=\lim_{m\to\infty}\|\mathscr F(g_m-g)\|_{L^2(\mathbb R^2)} \nonumber \\
&=\lim_{m\to\infty} \|g_m-g\|_{L^2(\mathbb R^2)}=0.\label{eq:stability-fourier-convergence}
\end{align}
It remains to show that
\begin{equation}
\label{eq:stability-dilation-convergence}
D_{\boldsymbol{b}^{(m)}}h_m
\longrightarrow
D_{\boldsymbol{b}}h
\quad\text{in }L^2(\mathbb R^2),
\end{equation}
where
$\boldsymbol{b}^{(m)}
:=(b_1^{(m)},b_2^{(m)}).$
Equation~\eqref{eq:stability-dilation-isometry} gives
\begin{align*}
&\|D_{\boldsymbol{b}^{(m)}}h_m-D_{\boldsymbol{b}}h\|_{L^2(\mathbb R^2)}\\
&\quad\leq\|D_{\boldsymbol{b}^{(m)}}(h_m-h)\|_{L^2(\mathbb R^2)}
+\|D_{\boldsymbol{b}^{(m)}}h-D_{\boldsymbol{b}}h\|_{L^2(\mathbb R^2)}\\
&\quad=\|h_m-h\|_{L^2(\mathbb R^2)}+
\|D_{\boldsymbol{b}^{(m)}}h-D_{\boldsymbol{b}}h\|_{L^2(\mathbb R^2)}.
\end{align*}
The first term tends to zero by
\eqref{eq:stability-fourier-convergence}. To treat the second term,
observe that the multiplication by a quadratic chirp preserves the
Schwartz space and that the Fourier transform maps
$\mathscr{S}(\mathbb R^2)$ onto itself. Hence, $h\in\mathscr{S}(\mathbb R^2).$
Since $\boldsymbol{b}^{(m)}\to\boldsymbol{b}$ and $b_1b_2\ne0$,
we deduce that there exist constants $\delta,M>0$ such that, for all sufficiently
large $m$,
$\delta\leq|b_k^{(m)}|\leq M,\ k\in\{1,2\}.$
Thus,
$D_{\boldsymbol{b}^{(m)}}h(\boldsymbol{u})
\to D_{\boldsymbol{b}}h(\boldsymbol{u})$
for every $\boldsymbol{u}\in\mathbb R^2$. The bound
$|b_k^{(m)}|\leq M$ also implies
$|(\frac{u_1}{b_1^{(m)}},
\frac{u_2}{b_2^{(m)}})|\geq
\frac{|\boldsymbol{u}|}{M}.$
Hence, for any fixed integer $N>1$, the rapid decay of $h$ gives
\begin{align*}
|D_{\boldsymbol{b}^{(m)}}h(\boldsymbol{u})|
\leq\delta^{-1}C_N\left[1+
\left|\left(\frac{u_1}{b_1^{(m)}},
\frac{u_2}{b_2^{(m)}}\right)
\right|\right]^{-N}
\leq C_N'(1+|\boldsymbol{u}|)^{-N},
\end{align*}
uniformly for all sufficiently large \(m\). The same estimate holds
for \(D_{\boldsymbol{b}}h\). Therefore,
\begin{align*}
|D_{\boldsymbol{b}^{(m)}}h(\boldsymbol{u})
-D_{\boldsymbol{b}}h(\boldsymbol{u})|^2&\leq
2|D_{\boldsymbol{b}^{(m)}}h(\boldsymbol{u})|^2
+2|D_{\boldsymbol{b}}h(\boldsymbol{u})|^2\\&\leq C_N''(1+|\boldsymbol{u}|)^{-2N}.
\end{align*}
For $N>1$,
$(1+|\boldsymbol{u}|)^{-2N}\in L^1(\mathbb R^2).$
The Lebesgue dominated convergence theorem therefore yields
\begin{align*}
\lim_{m\to\infty}
\|D_{\boldsymbol{b}^{(m)}}h-D_{\boldsymbol{b}}h\|_{L^2(\mathbb R^2)}^2=0.\end{align*}
Thus, $\lim_{m\to\infty}
\|D_{\boldsymbol{b}^{(m)}}h-D_{\boldsymbol{b}}h\|_{L^2(\mathbb R^2)}
=0.$
This proves \eqref{eq:stability-dilation-convergence}.
Define
$v_m:=D_{\boldsymbol{b}^{(m)}}h_m$ and
$v:=D_{\boldsymbol{b}}h.$
Then \(v_m\to v\) in \(L^2(\mathbb R^2)\) as $m\to\infty$. Since
$\frac{d_k^{(m)}}{b_k^{(m)}}
\longrightarrow\frac{d_k}{b_k},\
k\in\{1,2\},$ we also infer that
$\psi_{\boldsymbol{A}^{(m)}}(\boldsymbol{u})
\longrightarrow\psi_{\boldsymbol{A}}(\boldsymbol{u})$ as $m\to\infty$ for every $\boldsymbol{u}\in\mathbb R^2$. Consequently,
\begin{align*}
&\|e^{i\psi_{\boldsymbol{A}^{(m)}}}v_m
-e^{i\psi_{\boldsymbol{A}}}v\|_{L^2(\mathbb R^2)}\\
&\quad\leq\|v_m-v\|_{L^2(\mathbb R^2)}
+\|\bigl(e^{i\psi_{\boldsymbol{A}^{(m)}}}
-e^{i\psi_{\boldsymbol{A}}}\bigr)v
\|_{L^2(\mathbb R^2)} \to 0
\end{align*}
as $m\to\infty$. Indeed, since
$\left|e^{i\psi_{\boldsymbol{A}^{(m)}}}
-e^{i\psi_{\boldsymbol{A}}}\right|
\leq2$
and $v\in L^2(\mathbb R^2)$, the Lebesgue dominated convergence theorem then gives the above desired limit. 
Finally, under the principal-branch convention,
$\frac{|b_k|^{1/2}}{(ib_k)^{1/2}}
=e^{-i\pi\operatorname{sgn}(b_k)/4}.$
Since \(b_k^{(m)}\to b_k\ne0\), \(b_k^{(m)}\) and \(b_k\) have the
same sign for all sufficiently large \(m\). Hence,
$\chi\bigl(\boldsymbol{b}^{(m)}\bigr)
=\chi(\boldsymbol{b})$
for all sufficiently large \(m\). Combining this identity with
\eqref{eq:stability-lct-decomposition}, we obtain
\begin{align*}
&\|\mathscr{L}_{\boldsymbol{A}^{(m)}}(f)
-\mathscr{L}_{\boldsymbol{A}}(f)
\|_{L^2(\mathbb R^2)}\\
&\quad\leq\|e^{i\psi_{\boldsymbol{A}^{(m)}}}v_m
-e^{i\psi_{\boldsymbol{A}}}v
\|_{L^2(\mathbb R^2)}
+\left|\chi\left(\boldsymbol{b}^{(m)}
\right)-\chi(\boldsymbol{b})
\right|	\|v\|_{L^2(\mathbb R^2)}\\
&\quad\to0.
\end{align*}
Since $\Omega$ is a metric subspace of
$(\mathbb{R}^{2\times2})^2$, the established sequential convergence
is equivalent to continuity of
$\boldsymbol{A}\mapsto\mathscr{L}_{\boldsymbol{A}}(f)$ on $\Omega$.
This completes the proof.
\end{proof}

\section{Finite-Grid LCT Realization and Parameter Sensitivity}\label{bb}

The LCRT results used in Section~\ref{MF-B} are taken from~\cite{Yang2026LCRTImageEdgeDetection} and are not repeated here. This appendix records the finite-grid LCT realization with a parameter-sensitivity bound.

For one tensor axis of length $N$ and for any $n,m\in\{0,\ldots,N-1\},$ let
\begin{align*}
 x_n:=\left(n-\left\lfloor\frac N2\right\rfloor\right)\Delta x
 \quad \mathrm{and}\quad 
 \xi_m:=\frac{2\pi\left(m-\left\lfloor N/2\right\rfloor\right)}{N\Delta x}.
\end{align*}
Let $F_N$ be the unitary DFT on these centered grids, and define
$$D_x(r):=\operatorname{diag}(e^{irx_n^2/2})\quad\text{and}\quad D_\xi(r):=\operatorname{diag}(e^{ir\xi_m^2/2}).$$
For $A\in\mathrm{SL}(2,\mathbb R)$ with $b_A\neq0$, write
$$A:=\begin{pmatrix}
	a_A&b_A\\
	c_A&d_A
\end{pmatrix},$$
$$p_A:=\frac{a_A-1}{b_A}, \ {\rm and}\ 
  q_A:=\frac{d_A-1}{b_A}. $$
Then define
\begin{equation}\label{eq:finite_grid_lct}
T_A:=D_x(q_A)F_N^*D_\xi(-b_A)F_ND_x(p_A)\ \mathrm{and}\ T_{\boldsymbol A}:=T_{A_2}\otimes T_{A_1}.
\end{equation}
For a second admissible matrix
$B=\begin{pmatrix}a_B&b_B\\ c_B&d_B\end{pmatrix}$ with $b_B\neq0$,
we use the analogous notation
$p_B:=\frac {a_B-1}{b_B}$, $q_B:=\frac {d_B-1}{b_B}$, and $T_B$ as in \eqref{eq:finite_grid_lct} with $A$ replaced by $B$. In two dimensions,
$\boldsymbol B=(B_1,B_2)$ and $T_{\boldsymbol B}:=T_{B_2}\otimes T_{B_1}$.

\begin{proposition}\label{prop:finite-grid-realization}
Each factor in \eqref{eq:finite_grid_lct} is unitary. Consequently,
\begin{equation*}
T_A^{-1}=T_A^*\quad and\quad
T_{\boldsymbol A}^{-1}=T_{\boldsymbol A}^*.
\end{equation*}
For the grids defined above, let
\begin{equation*}
x_{\max}:=\max_{0\le n<N}|x_n|
\quad and \quad
\xi_{\max}:=\max_{0\le m<N}|\xi_m|.
\end{equation*}
Then, for any two admissible matrices $A$ and $B$ on the same grid,
\begin{align}\label{eq:finite_grid_sensitivity}
&\|T_A-T_B\|_2\notag\\
&\quad\le\frac12\left[x_{\max}^2\bigl(|p_A-p_B|+|q_A-q_B|\bigr)
+\xi_{\max}^2|b_A-b_B|\right].
\end{align}
The tensor-product realization satisfies the corresponding two-dimensional estimate. Moreover, for the real--imaginary readout in \eqref{eq:transform_readouts},
\begin{equation}\label{eq:finite_grid_readout_sensitivity}
\left\|\rho_{\mathrm L}^{(s)}(T_{\boldsymbol A}z)-\rho_{\mathrm L}^{(s)}(T_{\boldsymbol B}z)\right\|_2
\le \|W_{\mathrm L}^{(s)}\|_2\,\|T_{\boldsymbol A}-T_{\boldsymbol B}\|_2\,\|z\|_2,
\end{equation}
where $W_{\mathrm L}^{(s)}$ is the linear part of $P_{\mathrm L}^{(s)}$.
\end{proposition}

\begin{proof}
Each factor in \eqref{eq:finite_grid_lct} is unitary. Telescoping \eqref{eq:finite_grid_lct} and using $|e^{it}-e^{is}|\le|t-s|$ give \eqref{eq:finite_grid_sensitivity}, while tensor-product telescoping gives the two-dimensional estimate. Norm preservation under real--imaginary concatenation, together with cancellation of the projection bias in a difference, yields \eqref{eq:finite_grid_readout_sensitivity}. When the numerical floor in Subsection~\ref{FR-A} is inactive, this realization agrees with the implementation. The hard floor gives its piecewise-differentiable numerical counterpart, which completes the proof.
\end{proof}

\bibliographystyle{assets/IEEEtran}
\bibliography{assets/referencesIEEE}

\end{document}